\documentclass[11pt]{amsart}

\usepackage[letterpaper,margin=1in]{geometry}
\usepackage[T1]{fontenc}
\usepackage{lmodern}
\usepackage{microtype}
\usepackage{amsmath,amssymb,amsthm,mathtools}
\usepackage{enumitem}
\usepackage{aliascnt}   
\usepackage{mathrsfs}
\usepackage{etoolbox}
\usepackage{hyperref}
\usepackage[nameinlink,capitalise,noabbrev]{cleveref}

\hypersetup{
  colorlinks=true,
  linkcolor=blue,
  citecolor=blue,
  urlcolor=blue,
  pdftitle={A proof of the noncommutative discrete prime maximal inequality},
  pdfauthor={}
}

\numberwithin{equation}{section}
\setlist[itemize]{leftmargin=1.6em,itemsep=2pt,topsep=4pt}
\setlist[enumerate]{leftmargin=1.8em,itemsep=2pt,topsep=4pt}
\allowdisplaybreaks
\newtheorem{theorem}{Theorem}[section]

\newaliascnt{proposition}{theorem}
\newtheorem{proposition}[proposition]{Proposition}
\aliascntresetthe{proposition}
\crefname{proposition}{Proposition}{Propositions}
\Crefname{proposition}{Proposition}{Propositions}

\newaliascnt{lemma}{theorem}
\newtheorem{lemma}[lemma]{Lemma}
\aliascntresetthe{lemma}
\crefname{lemma}{Lemma}{Lemmas}
\Crefname{lemma}{Lemma}{Lemmas}

\newaliascnt{corollary}{theorem}
\newtheorem{corollary}[corollary]{Corollary}
\aliascntresetthe{corollary}
\crefname{corollary}{Corollary}{Corollaries}
\Crefname{corollary}{Corollary}{Corollaries}

\newaliascnt{remark}{theorem}
\newtheorem{remark}[remark]{Remark}
\aliascntresetthe{remark}
\crefname{remark}{Remark}{Remarks}
\Crefname{remark}{Remark}{Remarks}

\newaliascnt{assumption}{theorem}

\aliascntresetthe{assumption}
\crefname{assumption}{Assumption}{Assumptions}
\Crefname{assumption}{Assumption}{Assumptions}

\newcommand{\Acal}{\mathcal A}
\newcommand{\Mcal}{\mathcal M}
\newcommand{\Ncal}{\mathcal N}
\newcommand{\Tcal}{\mathcal T}
\newcommand{\Rcal}{\mathcal R}
\newcommand{\Pcal}{\mathcal P}

\newcommand{\supp}{\operatorname{supp}}
\newcommand{\Tr}{\operatorname{Tr}}
\newcommand{\e}{\mathrm e}
\newcommand{\1}{\mathbf 1}
\newcommand{\norm}[1]{\left\lVert #1\right\rVert}

\newcommand{\Lp}[2]{L_{#1}\!\left(#2\right)}
\newcommand{\Lpinf}[2]{L_{#1}\!\left(#2;\ell_\infty\right)}

\newcommand{\vphi}{\varphi}
\newcommand{\Z}{\mathbb Z}
\newcommand{\T}{\mathbb T}

\newcommand{\R}{\mathbb R}

\begin{document}
\title[Noncommutative prime ergodic averages]
{Pointwise convergence of noncommutative ergodic averages along the primes}

\author[]{Guixiang Hong}
\address{
Institute for Advanced Study in Mathematics\\
Harbin Institute of Technology\\
Harbin
150001\\
China}
\email{gxhong@hit.edu.cn}

\author[]{Liang Wang}
\address{Department of Mathematics\\ City University of Hong Kong\\Hong Kong SAR}
\email{L.Wang@cityu.edu.hk}

\thanks{}



\date{}

\begin{abstract}
Let $\mathcal N$ be a von Neumann algebra equipped with a normal faithful
semifinite trace, and let $\gamma$ be a trace-preserving automorphism of
$\mathcal N$. We consider the ergodic averages along the prime numbers
\[
  A_N(x)
  :=
  \frac1{|P_N|}
  \sum_{q\in P_N}\gamma^q(x),
  \qquad
  P_N:=\{q\leq N:q\ \text{is prime}\}.
\]
For every $1<p<\infty$, we prove a strong maximal inequality for
$(A_N)_{N\geq2}$ on $L_p(\mathcal N)$ and  that $A_N(x)$
converges bilaterally almost uniformly for every
$x\in L_p(\mathcal N)$. 

The proof exploits the circle method and a noncommutative sampling principle. For the convergence result, Bourgain's commutative argument uses pointwise maximal functions
and exceptional sets. These tools are not available in the
noncommutative setting. Instead, we show that the tails of the ergodic averages tend to zero in
$L_2(\mathcal N;\ell_\infty)$ and that the difference from the limit
belongs to $L_2(\mathcal N;c_0)$. This gives the desired b.a.u.
convergence, and provides a positive answer to one question left open in \cite{ChenHongWang+arXiv2024}. 
\end{abstract}

\maketitle


\section{Introduction}

In the previous paper \cite{ChenHongWang+arXiv2024}, the authors initiated
the study of noncommutative analogues of Bourgain's seminal works
\cite{BourgainMaximal,BourgainLp,BourgainArithmetic} on ergodic averages
along arithmetic sequences. However, only a maximal ergodic inequality in
the range 
\[
  p>\frac{1+\sqrt5}{2}
\]
is available, and two problems were left open. The first problem is whether the
noncommutative maximal ergodic inequality could be extended to all $p>1$,
as in Bourgain's work \cite{BourgainArithmetic}. The second was to establish
a noncommutative version of Bourgain's pointwise ergodic theorem
\cite{BourgainMaximal}. The main difficulty comes from the absence of an
underlying notion of \emph{points} in the noncommutative setting. As a
result, Bourgain's arguments based on \emph{maximal functions} and
\emph{exceptional sets} do not appear to transfer directly to the
noncommutative setting.

In the present paper, we solve the
first problem for the sequence of primes, and the second one completely. Let $\mathcal N$ be a von Neumann
algebra equipped with a normal semifinite faithful trace $\tau$, and let
$\gamma$ be a trace-preserving automorphism of $\mathcal N$; that is,
\[
  \tau\circ\gamma=\tau.
\]
For $1\leq p\leq\infty$, let $L_p(\mathcal N)$ be the associated
noncommutative $L_p$ space. The automorphism $\gamma$ extends to an isometry
on every $L_p(\mathcal N)$. We consider the prime averages
\begin{equation*}
  A_N(\cdot)
  :=
  \frac1{|P_N|}
  \sum_{q\in P_N}\gamma^q(\cdot),
  \qquad N\geq2,
\end{equation*}
where
\[
  P_N:=\{q\leq N:q\ \text{is prime}\}.
\]

Our main theorem gives both the maximal inequality and bilaterally almost
uniform convergence in the full range $1<p<\infty$. Any notation not
defined here will be explained in the next section.

\begin{theorem}\label{thm:prime-bau-convergence}
Let $1<p<\infty$. There exists a constant $C_p>0$, depending only on
$p$, such that
\begin{equation*}
  \left\|
    \bigl(A_N(x)\bigr)_{N\geq2}
  \right\|_{L_p(\mathcal N;\ell_\infty)}
  \leq C_p\|x\|_{L_p(\mathcal N)}
\end{equation*}
for every $x\in L_p(\mathcal N)$. Moreover, for every
$x\in L_p(\mathcal N)$, there exists $y\in L_p(\mathcal N)$ such that
\[
  A_N(x)\longrightarrow y
  \qquad\text{b.a.u. as }N\to\infty.
\]
\end{theorem}

The range $1<p<\infty$ is optimal. Indeed, LaVictoire \cite{Lavictoire} proved that, already in the
commutative setting, the weak type $(1,1)$ maximal inequality and the
pointwise ergodic theorem for the prime averages fail in general on
$L_1$.

As in the classical case, Theorem \ref{thm:prime-bau-convergence} can be deduced from the corresponding result for the
logarithmically weighted family
\begin{equation*}
  B_N(\cdot)
  :=
  \frac1N
  \sum_{q\in P_N}(\log q)\gamma^q(\cdot),
  \qquad N\geq2.
\end{equation*}

\begin{theorem}\label{thm:weighted-prime-maximal-Lp}
Let $1<p<\infty$.  There exists a constant $C_p>0$, depending only on
$p$, such that
\begin{equation*}
  \left\|
    \bigl(B_N(x)\bigr)_{N\geq2}
  \right\|_{L_p(\mathcal N;\ell_\infty)}
  \leq C_p\|x\|_{L_p(\mathcal N)}
\end{equation*}
for every $x\in L_p(\mathcal N)$.  Moreover, for every
$x\in L_p(\mathcal N)$ there exists $y\in L_p(\mathcal N)$ such that
\[
  B_N(x)\longrightarrow y
  \qquad\text{b.a.u. as }N\to\infty.
\]
\end{theorem}

In the commutative setting, Wierdl \cite{Wierdl} established the
equivalence between the convergence of the logarithmically weighted
and unweighted prime averages. The passage from the weighted averages to the unweighted averages for
maximal and variational estimates was later studied in
\cite[Section~5]{MTZK} and \cite[Appendix~A]{EisnerLin}.
In \Cref{sec:proof-main-theorem}, we adapt these classical arguments
to the noncommutative setting. For the maximal inequality, we combine
Abel summation with the characterization of
$L_p(\mathcal N;\ell_\infty)$. For convergence, we use the b.a.u. version of the Toeplitz lemma in
\Cref{lem:bau-toeplitz} to show that the weighted and unweighted
averages have the same limit.

To obtain the maximal inequalities in Theorem \ref{thm:weighted-prime-maximal-Lp}, we follow the strategy in our previous paper \cite{ChenHongWang+arXiv2024} by exploiting the sampling principle (\Cref{lem:sampling}). This considerably simplifies the classical proof in \cite{Wierdl}, see \Cref{rem:advantage-sampling} for more information. The reason why a similar argument as in \cite{ChenHongWang+arXiv2024} allows to get the maximal inequalities for all $1<p\leq \infty$ is an improved approximation estimate \eqref{eq:approximation-error}, and this estimate was established by Wierdl in \cite[(22)]{Wierdl} based on some results from number theory (see e.g. Proposition in \cite[Page 328]{Wierdl}).

To obtain the pointwise convergence result of Theorem \ref{thm:weighted-prime-maximal-Lp}, new ingredients are required compared to commutative argument. Indeed, Bourgain's approach involves not only a quantitative maximal inequality but also an argument based on pointwise defined maximal functions and exceptional sets. The latter argument is used to deduce almost everywhere convergence from quantitative maximal inequality via proof by contradiction. 
In the noncommutative setting, a similar quantitative maximal inequality is obtained in \Cref{prop:ergodic-grid-oscillation}, but the latter argument is not available. Our new observation is that the quantitative maximal inequality is strong enough so that it implies the $L_2(\mathcal M; \ell_\infty)$-norm of $(B_n(x)-B_N(x))_{n\geq N}$ converges to 0 as $N\rightarrow\infty$ for $x\in \mathcal{S}(\mathcal M)$, which in turn yields one $y\in L_2(\mathcal M)$ such that $(B_n(x)-y)_{n\geq 2}\in L_2(\mathcal M; c_0)$ and thus deduces the desired b.a.u. convergence. This result is stronger than the commutative one and thus  provides an alternative proof without involving the 'points'.  The
same method also applies to the polynomial averages studied in
\cite{ChenHongWang+arXiv2024}; see \Cref{rem:chw-pointwise-convergence}.

Note that for random sparse
averages, related b.a.u. convergence results were recently obtained by Le Merdy and Zadeh
\cite{LeMerdyZadeh} using probabilistic methods. But their approach does not apply to ergodic averages along primes or polynomials that are considered in the present paper.

The paper is organized as follows.  \Cref{pre} contains the needed facts
about noncommutative $L_p$ spaces, vector-valued maximal norms, b.a.u.
convergence, elementary number theory, and the major-arc approximation of
the prime multipliers.  In \Cref{sec:weighted-prime-maximal-proof}, we prove
the maximal part of \Cref{thm:weighted-prime-maximal-Lp}.  In
\Cref{sec:weighted-prime-bau-convergence}, we establish a noncommutative
Bourgain-type quantitative maximal inequality and use it to prove b.a.u.
convergence for the weighted averages.  Finally,
\Cref{sec:proof-main-theorem} passes from $B_N$ to $A_N$ and completes the
proof of \Cref{thm:prime-bau-convergence}.

\bigskip

{\bf Notation.}
Throughout the paper, $C$ denotes a positive absolute constant whose value
may change from one occurrence to the next.  When the dependence on
parameters is important, we write, for example, $C_{p,\varepsilon}$.  For
nonnegative quantities $X$ and $Y$, the notation $X\lesssim Y$ means that
$X\leq CY$, while $X\lesssim_{p,\varepsilon}Y$ allows the constant to depend
on the displayed parameters.  We write $X\simeq Y$ when both
$X\lesssim Y$ and $Y\lesssim X$ hold.  In contrast,
\[
  X_N\sim Y_N
  \quad\text{means}\quad
  \lim_{N\to\infty}\frac{X_N}{Y_N}=1.
\]
Thus $\sim$ includes the leading constant, whereas $\simeq$ does not.

\section{Preliminaries}\label{pre}

This section fixes the analytic and arithmetic notation used in the proof.
We first recall noncommutative $L_p$ spaces, then collect the elementary number-theoretic estimates needed
later, and finally describe the circle-method approximation of the prime
kernels.

\subsection{ Noncommutative \texorpdfstring{$L_p$}{Lp} spaces }
\label{subsec:nc-Lp}

Let $(\mathcal{M},\tau)$ be a semifinite von Neumann algebra equipped
with a normal semifinite faithful trace $\tau$. Denote by
$\mathcal{S}(\mathcal{M})_+$ the set of all $x\in\mathcal{M}_+$ such that
\[
  \tau\bigl(\supp(x)\bigr)<\infty,
\]
where $\supp(x)$ denotes the support projection of $x$. The linear span
of $\mathcal{S}(\mathcal{M})_+$ is denoted by
$\mathcal{S}(\mathcal{M})$. It is a weak$^*$-dense $*$-subalgebra of
$\mathcal{M}$.

For $1\le p<\infty$, define
\begin{equation}\label{eq:noncommutative-Lp-norm}
  \|x\|_p
  :=
  \bigl(\tau(|x|^p)\bigr)^{1/p},
  \qquad x\in\mathcal{S}(\mathcal{M}),
\end{equation}
where
\[
  |x|=(x^*x)^{1/2}
\]
is the modulus of $x$. The completion of
$\mathcal{S}(\mathcal{M})$ with respect to the norm in
\eqref{eq:noncommutative-Lp-norm} is called the noncommutative
$L_p$ space associated with $(\mathcal{M},\tau)$ and is denoted by
$L_p(\mathcal{M})$. As usual, we set
\[
  L_\infty(\mathcal{M})=\mathcal{M},
  \qquad
  \|x\|_\infty=\|x\|_{\mathcal{M}}.
\]
We refer to \cite{PX} for the basic theory of noncommutative
$L_p$ spaces.

\subsection{Noncommutative
\texorpdfstring{$\ell_\infty$}{l-infinity}-valued
\texorpdfstring{$L_p$}{Lp} spaces}
\label{subsec:vector-valued-Lp}

Let $1\le p\le\infty$, and let $\Lambda$ be an index set. The space
$L_p(\mathcal{M};\ell_\infty(\Lambda))$ (see e.g. \cite{Pisier+Asterique98, Junge+Crelle2002}) consists of all families
$x=(x_\lambda)_{\lambda\in\Lambda}\subset L_p(\mathcal{M})$ for which
there exist
\[
  a,b\in L_{2p}(\mathcal{M})
  \qquad\text{and}\qquad
  (y_\lambda)_{\lambda\in\Lambda}\subset L_\infty(\mathcal{M})
\]
such that
\begin{equation*}
  x_\lambda=ay_\lambda b,
  \qquad \lambda\in\Lambda,
\end{equation*}
and
\[
  \sup_{\lambda\in\Lambda}\|y_\lambda\|_\infty<\infty.
\]

The norm on $L_p(\mathcal{M};\ell_\infty(\Lambda))$ is defined by
\begin{equation*}
  \|x\|_{L_p(\mathcal{M};\ell_\infty(\Lambda))}
  :=
  \inf_{x_\lambda=ay_\lambda b}
  \left\{
    \|a\|_{2p}
    \sup_{\lambda\in\Lambda}\|y_\lambda\|_\infty
    \|b\|_{2p}
  \right\}.
\end{equation*}
It was shown in \cite{JX} that
$x\in L_p(\mathcal M;\ell_\infty(\Lambda))$ if and only if
\begin{equation*}
  \sup_{\substack{I\subset\Lambda\\ |I|<\infty}}
  \left\|
    (x_\lambda)_{\lambda\in I}
  \right\|_{L_p(\mathcal M;\ell_\infty(I))}
  <\infty.
\end{equation*}
Moreover,
\begin{equation}\label{eq:finite-subset-norm-identity}
  \left\|
    (x_\lambda)_{\lambda\in\Lambda}
  \right\|_{L_p(\mathcal M;\ell_\infty(\Lambda))}
  =
  \sup_{\substack{I\subset\Lambda\\ |I|<\infty}}
  \left\|
    (x_\lambda)_{\lambda\in I}
  \right\|_{L_p(\mathcal M;\ell_\infty(I))}.
\end{equation}

If every $x_\lambda$ is self-adjoint, then
$x\in L_p(\mathcal{M};\ell_\infty(\Lambda))$ if and only if there
exists $a\in L_p(\mathcal{M})_+$ such that
\[
  -a\le x_\lambda\le a,
  \qquad \lambda\in\Lambda.
\]
In this case,
\begin{equation*}
  \left\|
    (x_\lambda)_{\lambda\in\Lambda}
  \right\|_{L_p(\mathcal M;\ell_\infty(\Lambda))}
  =
  \inf
  \left\{
    \|a\|_p:
    a\in L_p(\mathcal M)_+,\,
    -a\le x_\lambda\le a
    \text{ for every }\lambda\in\Lambda
  \right\}.
\end{equation*}
When no confusion can arise, we abbreviate
$L_p(\mathcal{M};\ell_\infty(\Lambda))$ as
$L_p(\mathcal{M};\ell_\infty)$.

Finally, we recall the following interpolation identity.

\begin{lemma}\label{lem:maximal-interpolation}
Let $1\le p_0,p_1\le\infty$ and $0<\theta<1$. Define $p$ by
\begin{equation*}
  \frac{1}{p}
  =
  \frac{1-\theta}{p_0}
  +
  \frac{\theta}{p_1}.
\end{equation*}
Then
\begin{equation*}
  L_p(\mathcal{M};\ell_\infty)
  =
  \bigl(
    L_{p_0}(\mathcal{M};\ell_\infty),
    L_{p_1}(\mathcal{M};\ell_\infty)
  \bigr)_\theta
\end{equation*}
with equivalent norms.
\end{lemma}

The interpolation result in \Cref{lem:maximal-interpolation} can be
found in \cite{JX}.

\subsection{Almost uniform convergence and the space
\texorpdfstring{$L_p(\mathcal M;c_0)$}{Lp(M;c0)}}
\label{subsec:bau-c0}

We next recall the noncommutative analogues of almost everywhere
convergence; see e.g. \cite{JX} for more information.   

Let $L_0(\mathcal M)$ denote the space of all
$\tau$-measurable operators affiliated with $\mathcal M$.
A sequence $(x_n)_{n\ge1}\subset L_0(\mathcal M)$ is said to converge
\emph{bilaterally almost uniformly} to $x\in L_0(\mathcal M)$,
abbreviated as
\[
  x_n\longrightarrow x
  \qquad\text{b.a.u.},
\]
if, for every $\varepsilon>0$, there exists a projection
$e\in\mathcal M$ such that
\begin{equation*}
  \tau(e^\perp)<\varepsilon
  \qquad\text{and}\qquad
  \lim_{n\to\infty}\|e(x_n-x)e\|_\infty=0,
\end{equation*}
where $e^\perp=\mathbf 1-e$. 

For a sequence indexed by $\mathbb N$, the space
$L_p(\mathcal M;c_0)$ is the closure in
$L_p(\mathcal M;\ell_\infty)$ of the finite sequences.
Equivalently, $x=(x_n)_{n\ge1}$ belongs to $L_p(\mathcal M;c_0)$ if
and only if it admits a factorization
\begin{equation*}
  x_n=ay_n b,
  \qquad n\ge1,
\end{equation*}
where
\[
  a,b\in L_{2p}(\mathcal M),
  \qquad
  y_n\in L_\infty(\mathcal M),
  \qquad
  \lim_{n\to\infty}\|y_n\|_\infty=0.
\]
The relevance of this space to individual convergence is the implication
\begin{equation}\label{eq:c0-implies-bau}
  (x_n)_{n\ge1}\in L_p(\mathcal M;c_0)
  \quad\Longrightarrow\quad
  x_n\longrightarrow0
  \quad\text{b.a.u.}
\end{equation}
Consequently, if $x\in L_p(\mathcal M)$ and
$(x_n-x)_{n\ge1}\in L_p(\mathcal M;c_0)$, then $x_n\to x$ b.a.u.
This implication will be used repeatedly in
\Cref{sec:weighted-prime-bau-convergence}.

\subsection{Elementary number-theoretic estimates}
\label{subsec:number-theoretic-estimates}

We collect here the elementary arithmetic facts used throughout the paper.
Let
\begin{equation*}
  \pi(N):=\#\{q\le N:q\ \text{is prime}\},
  \qquad
  \vartheta(N):=\sum_{\substack{q\le N\\q\ \mathrm{prime}}}\log q.
\end{equation*}
The classical Chebyshev estimates imply that there exist absolute constants
$c,C>0$ such that
\begin{equation}\label{eq:chebyshev-pi-bounds}
  c\frac{N}{\log N}
  \le \pi(N)
  \le C\frac{N}{\log N},
  \qquad N\ge2,
\end{equation}
and
\begin{equation}\label{eq:chebyshev-theta-upper}
  \vartheta(N)\le CN,
  \qquad N\ge2.
\end{equation}
We shall also use the prime number theorem in either of the equivalent forms
\begin{equation}\label{eq:PNT-pi-form}
  \pi(N)\sim\frac{N}{\log N}
  \qquad\text{and}\qquad
  \vartheta(N)\sim N
  \quad\text{as }N\to\infty.
\end{equation}
See \cite[Chapters~4 and~6]{Apostol} or
\cite[Sections~22.1--22.2]{HW} for these facts.

We also record the precise asymptotic estimate
\begin{equation}\label{eq:reciprocal-log-estimates}
  \sum_{n=3}^N\frac1{\log n}
  \sim \frac{N}{\log N}.
\end{equation}
For example, this follows from the Stolz--Ces\`aro theorem.  The precise
leading constant is used in
\Cref{sec:proof-weighted-unweighted-equivalence}, while its weaker upper-bound
consequence is used in the maximal estimates below.

Finally, let $\mu$, $\vphi$, and $d$ denote the M\"obius function, Euler's
totient function, and the divisor-counting function, respectively.  Thus
\begin{equation}\label{eq:mobius-function}
  \mu(n)
  :=
  \begin{cases}
    1, & n=1,\\
    (-1)^k, & n \text{ is the product of $k$ distinct primes},\\
    0, & n \text{ is not square-free},
  \end{cases}
\end{equation}

\begin{equation}\label{eq:euler-totient-function}
  \vphi(n):=\#\{1\le a\le n:(a,n)=1\},
\end{equation}
and
\begin{equation*}
  d(n):=\#\{1\le d\le n:d|n\}.
\end{equation*}
For every $\varepsilon>0$,
\begin{equation}\label{eq:standard-divisor-totient-estimates}
  d(n)\lesssim_\varepsilon n^\varepsilon,
  \qquad
  \frac{n}{\vphi(n)}\lesssim_\varepsilon n^\varepsilon,
  \qquad n\ge1.
\end{equation}
 We refer to \cite[Chapter~18]{HW}.  The estimates in
\eqref{eq:standard-divisor-totient-estimates} are used in the summation of the
major-arc pieces in \Cref{lem:arithmetic-sum,prop:L2-major-arc-gain}.

\subsection{Exponential sums and approximate kernels}
\label{subsec:prime-kernels}

We identify $\mathbb T=\mathbb R/\mathbb Z$ with $[0,1]$, with the
endpoints identified. Differences in $\mathbb T$ are taken modulo $1$
and represented in $[-1/2,1/2)$.

For $N\geq2$, define the positive scalar kernel
\begin{equation*}
  K_N(m)
  :=
  \frac{\log m}{N}
  \1_{\{1\le m\le N,\ m\ \mathrm{prime}\}},
  \qquad m\in\mathbb{Z}.
\end{equation*}
The corresponding convolution operator is given by
\[
  K_N*f(n)
  =
  \frac{1}{N}
  \sum_{\substack{1\le m\le N\\ m\ \mathrm{prime}}}
  (\log m)f(n-m).
\]
Let $(\Mcal,\tau)$ be a semifinite von Neumann algebra.  We write
\[
 \Acal=\ell_\infty(\Z)\bar\otimes\Mcal,
 \qquad \Tr_{\Acal}=\#\otimes\tau,
\]
where $\#$ is the count measure.
The discrete maximal estimate that will be established in the following is
\begin{equation*}
  \bigl\|(K_N*f)_{N\ge2}\bigr\|_{\Lpinf p\Acal}
  \le
  C_p\norm{f}_{\Lp p\Acal},
  \qquad 1<p<\infty.
\end{equation*}

For $N\ge1$, let
\begin{equation*}
  \nu_N(\xi)
  :=
  \frac{1}{N}\sum_{m=1}^N \e^{-2\pi i m\xi},
  \qquad \xi\in\mathbb{T}.
\end{equation*}
Choose a smooth cutoff function $\eta$ supported in $(-1/2,1/2)$,
and a slightly larger smooth cutoff function $\widetilde\eta$ such that
\begin{align}\label{eta}
  0\le \widetilde\eta\le1
  \qquad\text{and}\qquad
  \widetilde\eta=1
  \quad\text{on }\supp \eta.
\end{align}
For $s\ge1$, set
\[
  Q_s:=2^{s+1}.
\]
Choose $D_s$ sufficiently large, and define the Fourier multiplier
operators $\Rcal_{q,s,N}$ and $\Tcal_{s,N}$ by
\begin{equation*}
  \widehat{\Rcal_{q,s,N}f}(\xi)
  :=
  \sum_{\substack{1\le a\le q\\(a,q)=1}}
  \nu_N\!\left(\xi-\frac{a}{q}\right)
  \eta\!\left(D_s\!\left(\xi-\frac{a}{q}\right)\right)
  \widehat f(\xi)
\end{equation*}
and
\begin{equation*}
  \Tcal_{s,N}f
  :=
  \sum_{Q_s/2\le q<Q_s}
  \frac{\mu(q)}{\vphi(q)}\Rcal_{q,s,N}f.
\end{equation*}
Here $\mu$ and $\vphi$ are as in
\eqref{eq:mobius-function} and \eqref{eq:euler-totient-function}.

We choose $D_s$ so that the following two conditions
are satisfied:
\begin{enumerate}[label=\textup{(\roman*)}]
  \item $D_s\ge4Q_s$;

  \item the supports of
  \[
    \widetilde\eta\!\left(
      D_s\!\left(\,\cdot-\frac{a}{q}\right)
    \right),
    \qquad
    Q_s/2\le q<Q_s,
    \quad (a,q)=1,
  \]
  are pairwise disjoint.
\end{enumerate}

One can then construct approximate kernels $L_N$ of the form
\begin{equation}\label{eq:LN-decomposition}
  L_N
  =
  L_N^{(0)}+\sum_{s\ge1}\Tcal_{s,N},
\end{equation}
where the Fourier multiplier associated with $L_N^{(0)}$ is
\begin{equation*}
  \widehat{L_N^{(0)}}(\xi)
  =
  \nu_N(\xi)\eta(\xi).
\end{equation*}
Moreover, for every $A>0$, there exists a constant $C_A>0$ such that
\begin{equation}\label{eq:approximation-error}
  \norm{\widehat K_N-\widehat L_N}_{L_\infty(\mathbb T)}
  \leq
  C_A(\log N)^{-A},
  \qquad N\geq2.
\end{equation}
For the original construction and the proof of
\eqref{eq:approximation-error}, see
\cite[pp.~321--334]{Wierdl}.

\section{Maximal inequality for the weighted prime averages}
\label{sec:weighted-prime-maximal-proof}

The aim of this section is to prove the maximal assertion of
\Cref{thm:weighted-prime-maximal-Lp}. The key ingredient is the following
dyadic maximal inequality for the logarithmically weighted prime averages. All the unexplained notation appearing below can be found in the previous section.

\begin{proposition}\label{prop:dyadic-prime-maximal}
Let $1<p<2$. Then
\begin{equation}\label{eq:dyadic-prime-maximal}
  \bigl\|(K_{2^k}*f)_{k\ge1}\bigr\|_{\Lpinf p\Acal}
  \le
  C_p\norm{f}_{\Lp p\Acal}
\end{equation}
for every $f\in L_p(\Acal)$.
\end{proposition}

Inequality \eqref{eq:dyadic-prime-maximal} holds trivially for $2\leq p\leq\infty$ by interpolation with the obvious estimate \eqref{eq:full-discrete-prime-Linfty}. We first show that \Cref{prop:dyadic-prime-maximal} implies the maximal
assertion of \Cref{thm:weighted-prime-maximal-Lp}.

\begin{proof}[Proof of the maximal assertion in
\Cref{thm:weighted-prime-maximal-Lp}.]
We now remove the dyadic restriction and complete the transference to the
ergodic averages. Suppose first that $f\in L_p(\Acal)_+$ and
$2^{k-1}<N\le2^k$. Positivity gives
\begin{equation*}
  0\le K_N*f
  \le \frac{2^k}{N}K_{2^k}*f
  \le 2K_{2^k}*f.
\end{equation*}
The positive-majorant characterization of
$L_p(\Acal;\ell_\infty)$ and \Cref{prop:dyadic-prime-maximal} therefore yield
\begin{equation}\label{eq:full-discrete-prime-maximal-below-two}
  \bigl\|(K_N*f)_{N\ge2}\bigr\|_{\Lpinf p\Acal}
  \le C_p\norm{f}_{\Lp p\Acal},
  \qquad 1<p<2.
\end{equation}
The estimate for arbitrary $f$ follows by decomposing its real and imaginary
parts into their positive and negative parts.

It remains to cover $2\le p<\infty$. By
\eqref{eq:chebyshev-theta-upper},
\[
  \sup_{N\ge2}\norm{K_N}_{\ell_1(\mathbb Z)}<\infty,
\]
and hence
\begin{equation}\label{eq:full-discrete-prime-Linfty}
  \bigl\|(K_N*f)_{N\ge2}\bigr\|_{
    L_\infty(\Acal;\ell_\infty)}
  \le C\norm{f}_{L_\infty(\Acal)}.
\end{equation}
Interpolating \eqref{eq:full-discrete-prime-maximal-below-two}, for any fixed
exponent in $(1,2)$, with \eqref{eq:full-discrete-prime-Linfty}, and using
\Cref{lem:maximal-interpolation}, gives
\begin{equation}\label{eq:full-discrete-prime-maximal}
  \bigl\|(K_N*f)_{N\ge2}\bigr\|_{\Lpinf p\Acal}
  \le C_p\norm{f}_{\Lp p\Acal},
  \qquad 1<p<\infty.
\end{equation}

Finally, the standard noncommutative Calder\'on transference argument,
applied to the orbit of $x$ under $\gamma$ as in
\cite[Section~4.2]{ChenHongWang+arXiv2024}, transfers
\eqref{eq:full-discrete-prime-maximal} to
\[
  \bigl\|
    (B_N(x))_{N\ge2}
  \bigr\|_{L_p(\mathcal N;\ell_\infty)}
  \le
  C_p\|x\|_{L_p(\mathcal N)},
  \qquad 1<p<\infty.
\]
This proves the maximal assertion in
\Cref{thm:weighted-prime-maximal-Lp}.
\end{proof}

The rest of this section is devoted to the proof of
\Cref{prop:dyadic-prime-maximal}.  We use the approximate kernels $L_N$
arising from Wierdl's circle-method construction and the approximation
estimate \eqref{eq:approximation-error}.  The main new ingredient is a
noncommutative treatment of the rational major arcs.  More precisely, the
sampling principle provides estimates uniformly in the denominator, while
$L_2$ orthogonality and interpolation allow us to sum the contributions of the major-arc
.  A comparison with \cite[Lemma~$3'$]{Wierdl} is given in
\Cref{rem:advantage-sampling}.

\subsection{Maximal estimates for the approximate kernels}
\label{subsec:approximate-kernel-maximal}
We shall use the following more general form of the noncommutative
maximal sampling principle in \cite[Lemma~3.5]{ChenHongWang+arXiv2024}. We first introduce the required notation.

Let $q\geq1$ be an integer, and let $\phi$ be a smooth function on
$\mathbb R$ satisfying
\[
  \supp(\phi)
  \subset
  \left(-\frac{1}{2q},\frac{1}{2q}\right).
\]
Define the $1/q$-periodization of $\phi$ by
\[
  \phi_{\mathrm{per}}^q(\xi)
  :=
  \sum_{n\in\mathbb Z}
  \phi\left(\xi-\frac{n}{q}\right).
\]
Related to the Fourier multiplier $T_\phi$  on $L_2(\mathbb R)$ with symbol $\phi$, we introduce $(T_\phi^q)_{\mathrm{dis}}$ which is  the discrete Fourier multiplier
associated with the symbol $\phi_{\mathrm{per}}^q$; when $q=1$, we will omit $q$ and denote it as $(T_\phi)_{\mathrm{dis}}$ with symbol $\phi_{\mathrm{per}}$. Thus, for every
suitable function $f\colon\mathbb Z\to\mathcal M$,
\[
  \sum_{n\in\mathbb Z}
  \bigl((T_\phi^q)_{\mathrm{dis}}f\bigr)(n)
  \mathrm e^{-2\pi i n\xi}
  =
  \phi_{\mathrm{per}}^q(\xi)
  \sum_{n\in\mathbb Z}
  f(n)\mathrm e^{-2\pi i n\xi}.
\]
By the Fourier inversion formula,
\[
  \bigl((T_\phi^q)_{\mathrm{dis}}f\bigr)(n)
  =
  q\sum_{m\in\mathbb Z}
  f(n-mq)\,\phi^\vee(mq).
\]

\begin{lemma}\label{lem:sampling}
Let $1\leq p,r\leq\infty$, let $I$ be a finite set, and, for every
$i\in I$, let $\Lambda_i$ be a parameter set.  Suppose that
$\phi_{i,\lambda}\in C_c^\infty(\mathbb R)$ satisfies
\[
  \supp(\phi_{i,\lambda})
  \subset
  \left(-\frac12,\frac12\right),
  \qquad
  i\in I,\quad \lambda\in\Lambda_i.
\]
  Assume that there exists a constant $C>0$ such
that
\begin{equation*}
  \left(
    \sum_{i\in I}
    \left\|
      \bigl(
        T_{\phi_{i,\lambda}}g
      \bigr)_{\lambda\in\Lambda_i}
    \right\|_{
      L_p\bigl(
        L_\infty(\mathbb R)\overline{\otimes}\mathcal M;
        \ell_\infty
      \bigr)
    }^r
  \right)^{1/r}
  \leq
  C
  \|g\|_{
    L_p\bigl(
      L_\infty(\mathbb R)\overline{\otimes}\mathcal M
    \bigr)
  }
\end{equation*}
for every semifinite von Neumann algebra $\mathcal M$ and every $g\in
  L_p\bigl(
    L_\infty(\mathbb R)\overline{\otimes}\mathcal M
  \bigr).$
Then
\begin{equation*}
  \left(
    \sum_{i\in I}
    \left\|
      \bigl(
        (T_{\phi_{i,\lambda,q}}^q)_{\mathrm{dis}}f
      \bigr)_{\lambda\in\Lambda_i}
    \right\|_{
      L_p\bigl(
        L_\infty(\mathbb Z)\overline{\otimes}\mathcal M;
        \ell_\infty
      \bigr)
    }^r
  \right)^{1/r}
  \leq
  C_p C
  \|f\|_{
    L_p\bigl(
      L_\infty(\mathbb Z)\overline{\otimes}\mathcal M
    \bigr)
  }
\end{equation*}
for every $  f\in
  L_p\bigl(
    L_\infty(\mathbb Z)\overline{\otimes}\mathcal M
  \bigr)$, where 
$  \phi_{i,\lambda,q}(\xi)
  :=
  \phi_{i,\lambda}(q\xi)$ and
the constant $C_{p,r}$ depends only on $p$ and $r$.
\end{lemma}

\begin{proof}
When $I$ contains only one element, this is exactly
\cite[Lemma~3.5]{ChenHongWang+arXiv2024}.  We omit the details for
general $I$, since the proof is almost identical.  The only difference
is that we keep the outer $\ell_r$-norm throughout the argument.
\end{proof}

We now apply the sampling principle to the approximate kernels.  In view
of \eqref{eq:LN-decomposition}, we first estimate the low-frequency term
$L_N^{(0)}$ in \Cref{lem:low-frequency-maximal}, and then control the
rational major-arc pieces $\Tcal_{s,N}$ in
\Cref{lem:summation-over-major-arcs}.

\begin{lemma}\label{lem:low-frequency-maximal}
Let $1<p\le\infty$. Then
\begin{equation}\label{eq:low-frequency-maximal}
  \bigl\|(L_N^{(0)}*f)_{N\ge1}\bigr\|_{\Lpinf p\Acal}
  \le
  C_p\norm{f}_{\Lp p\Acal}.
\end{equation}
\end{lemma}

\begin{proof}
Let $B$ be the convolution operator on $\Acal$ whose symbol is the 1-periodization of $\eta$ in \eqref{eta}, and let $\sigma$ be the shift on $\Acal$ defined by
\[
  (\sigma f)(j):=f(j-1),
  \qquad j\in\Z.
\]
Since the Fourier transform of  $L_N^{(0)}$ is
$\nu_N\eta$, we have
\begin{equation*}
  L_N^{(0)}*f
  =
  \frac{1}{N}\sum_{m=1}^N\sigma^m(Bf).
\end{equation*}
The shift $\sigma$ is a trace-preserving $*$-automorphism of $\Acal$.
Therefore, the noncommutative maximal ergodic theorem (see \cite{JX}) gives
\begin{equation}\label{eq:shift-maximal-estimate}
  \left\|
    \left(
      \frac{1}{N}\sum_{m=1}^N\sigma^m(Bf)
    \right)_{N\ge1}
  \right\|_{\Lpinf p\Acal}
  \le
  C_p\norm{Bf}_{\Lp p\Acal}.
\end{equation}
Since $\eta$ is smooth, the convolution kernel of $B$ belongs to
$\ell_1(\Z)$. Hence $B$ is bounded on $\Lp p\Acal$, and
\eqref{eq:low-frequency-maximal} follows from
\eqref{eq:shift-maximal-estimate}.
\end{proof}

We next turn to the rational arcs.  It is convenient to begin by summing over
all residue classes modulo $q$.  This removes the coprimality condition and
places the multiplier exactly in the form covered by
\Cref{lem:sampling}.

For $D\ge4q$, define the Fourier multiplier operator
$\Pcal_{q,D,N}$ by
\begin{equation}\label{eq:all-residue-multiplier}
  \widehat{\Pcal_{q,D,N}f}(\xi)
  :=
  \sum_{b=0}^{q-1}
  \nu_N\!\left(\xi-\frac{b}{q}\right)
  \eta\!\left(D\!\left(\xi-\frac{b}{q}\right)\right)
  \widehat f(\xi).
\end{equation}

\begin{lemma}\label{lem:all-residues}
Let $1<p<\infty$. Then
\begin{equation}\label{eq:all-residues-maximal}
  \bigl\|(\Pcal_{q,D,N}f)_{N\ge1}\bigr\|_{\Lpinf p\Acal}
  \le
  C_p\norm{f}_{\Lp p\Acal},
\end{equation}
where $C_p$ is independent of $q$ and $D$.
\end{lemma}

\begin{proof}
For $\theta\in\R$, define
\begin{equation*}
  \Phi_{q,D,N}(\theta)
  :=
  \nu_N(\theta/q)\eta(D\theta/q).
\end{equation*}
Since $\supp(\eta)\subset(-1/2,1/2)$ and $D\ge4q$, we have
\[
  \supp(\Phi_{q,D,N})
  \subset
  \left(-\frac{q}{2D},\frac{q}{2D}\right)
  \subset
  \left(-\frac12,\frac12\right).
\]

Let $B_{q,D}$ be the Fourier multiplier on
$L_p(L_\infty(\R)\overline{\otimes}\Mcal)$ with symbol
$\eta(D\theta/q)$. Its convolution kernel is a dilation of
$\check\eta$, and consequently
\begin{equation}\label{eq:BqD-bound}
  \norm{B_{q,D}g}_{\Lp p{L_\infty(\R)\overline{\otimes}\Mcal}}
  \le
  \norm{\check\eta}_{L_1(\R)}
  \norm{g}_{\Lp p{L_\infty(\R)\overline{\otimes}\Mcal}}.
\end{equation}

Let $\sigma_{1/q}$ denote translation by $1/q$ on
$L_\infty(\R)\overline{\otimes}\Mcal$; namely,
\[
  (\sigma_{1/q}g)(x)
  :=
  g\!\left(x-\frac{1}{q}\right).
\]
The Fourier multiplier with symbol $\Phi_{q,D,N}$ can then be written as
\begin{equation*}
  T_{\Phi_{q,D,N}}g
  =
  \frac{1}{N}\sum_{m=1}^N\sigma_{1/q}^m(B_{q,D}g).
\end{equation*}
Since $\sigma_{1/q}$ is a trace-preserving $*$-automorphism, the
noncommutative maximal ergodic theorem, together with
\eqref{eq:BqD-bound}, yields
\begin{equation}\label{eq:continuous-Phi-maximal}
  \bigl\|(T_{\Phi_{q,D,N}}g)_{N\ge1}\bigr\|_{
    \Lpinf p{L_\infty(\R)\overline{\otimes}\Mcal}}
  \le
  C_p
  \norm{g}_{\Lp p{L_\infty(\R)\overline{\otimes}\Mcal}}.
\end{equation}
We next identify the discrete multiplier obtained from the sampling lemma.
In the notation of \Cref{lem:sampling},
define
\begin{equation*}
\begin{aligned}
  \Phi_{q,D,N}^{[q]}(\xi)
  =\Phi_{q,D,N}(q\xi)
  =\nu_N(\xi)\eta(D\xi).
\end{aligned}
\end{equation*}
The $1/q$-periodization of $\Phi_{q,D,N}^{[q]}$ is then
\begin{equation*}
\begin{aligned}
  \bigl(\Phi_{q,D,N}^{[q]}\bigr)_{\mathrm{per}}^q(\xi)
  =\sum_{n\in\mathbb Z}
    \Phi_{q,D,N}^{[q]}\!\left(\xi-\frac nq\right)=\sum_{n\in\mathbb Z}\Phi_{q,D,N}(q\xi-n).
\end{aligned}
\end{equation*}
Write every $n\in\mathbb{Z}$ uniquely in the form
\[
  n=a+q\ell,
  \qquad 0\le a\le q-1,
  \quad \ell\in\mathbb Z.
\]
Using the $1$-periodicity of $\nu_N$, we obtain
\begin{equation*}
\begin{aligned}
  \bigl(\Phi_{q,D,N}^{[q]}\bigr)_{\mathrm{per}}^q(\xi)
  &=\sum_{a=0}^{q-1}\sum_{\ell\in\mathbb Z}
    \Phi_{q,D,N}(q\xi-a-q\ell)\\
  &=\sum_{a=0}^{q-1}\sum_{\ell\in\mathbb Z}
    \nu_N\!\left(\xi-\frac aq-\ell\right)
    \eta\!\left(D\left(\xi-\frac aq-\ell\right)\right)\\
  &=\sum_{a=0}^{q-1}
    \nu_N\!\left(\xi-\frac aq\right)
    \sum_{\ell\in\mathbb Z}
    \eta\!\left(D\left(\xi-\frac aq-\ell\right)\right)\\
  &=\sum_{a=0}^{q-1}
    \nu_N\!\left(\xi-\frac aq\right)
    \eta\!\left(D\left(\xi-\frac aq\right)\right).
\end{aligned}
\end{equation*}
This is precisely the multiplier defining $\Pcal_{q,D,N}$ in
\eqref{eq:all-residue-multiplier}.  Therefore,
\begin{equation}\label{eq:sampled-operator-identification}
  (T_{\Phi_{q,D,N}^{[q]}}^q)_{\mathrm{dis}}
  =\Pcal_{q,D,N}.
\end{equation}
Applying \Cref{lem:sampling} to $(\Phi_{q,D,N})_{N\ge1}$ and using
\eqref{eq:continuous-Phi-maximal} and
\eqref{eq:sampled-operator-identification}, we conclude that
\[
  \bigl\|(\Pcal_{q,D,N}f)_{N\ge1}\bigr\|_{\Lpinf p\Acal}
  \le C_{p,\eta}\norm{f}_{\Lp p\Acal}.
\]
This proves \eqref{eq:all-residues-maximal}.
\end{proof}

\begin{remark}\label{rem:advantage-sampling}
\normalfont
The uniformity in $q$ in \Cref{lem:all-residues} is an important
advantage of our approach.  In Wierdl's scalar proof
\cite[Lemma~$3'$]{Wierdl}, the corresponding estimate is obtained through
a delicate decomposition into residue classes and a repeated comparison
between continuous and discrete norms.  By contrast, the sampling principle
in \Cref{lem:sampling} gives the required estimate directly, with a constant
independent of both $q$ and the ambient von Neumann algebra.  This
considerably simplifies the treatment of the rational translates, avoids
the delicate normalization arising in the residue-class decomposition, and
applies at the same time to operator-valued functions.
\end{remark}

\begin{lemma}\label{lem:reduced-residues}
Let $1<p<\infty$. Then
\begin{equation*}
  \bigl\|(\Rcal_{q,s,N}f)_{N\ge1}\bigr\|_{\Lpinf p\Acal}
  \le
  C_p d(q)\norm{f}_{\Lp p\Acal},
\end{equation*}
where $d(q)$ denotes the number of positive divisors of $q$.
\end{lemma}

\begin{proof}
The M\"obius inversion identity (see \cite[Chapter~2]{Apostol})
\begin{equation*}
  \sum_{\substack{1\le a\le q\\(a,q)=1}}F(a/q)
  =
  \sum_{d\mid q}\mu(d)
  \sum_{b=0}^{q/d-1}
  F\!\left(\frac{b}{q/d}\right)
\end{equation*}
implies that
\begin{equation}\label{eq:R-as-P-sum}
  \Rcal_{q,s,N}
  =
  \sum_{d\mid q}\mu(d)\Pcal_{q/d,D_s,N}.
\end{equation}
Since $D_s\ge4q\ge4(q/d)$, \Cref{lem:all-residues},
\eqref{eq:R-as-P-sum}, and the triangle inequality give
\begin{equation*}
\begin{aligned}
  \bigl\|(\Rcal_{q,s,N}f)_{N\ge1}\bigr\|_{\Lpinf p\Acal}
  &\le
  \sum_{d\mid q}
  \bigl\|(\Pcal_{q/d,D_s,N}f)_{N\ge1}\bigr\|_{\Lpinf p\Acal}\\
  &\le
  C_p d(q)\norm{f}_{\Lp p\Acal}.
\end{aligned}
\end{equation*}
\end{proof}

To sum the reduced-residue estimates over denominators of comparable size,
we use the elementary bounds collected in
\Cref{subsec:number-theoretic-estimates}.

\begin{lemma}\label{lem:arithmetic-sum}
For every $\varepsilon>0$, there exists a constant $C_\varepsilon>0$
such that
\begin{equation}\label{eq:arithmetic-sum}
  \sum_{Q/2\le q<Q}\frac{d(q)}{\vphi(q)}
  \le
  C_\varepsilon Q^\varepsilon,
  \qquad Q\ge2.
\end{equation}
\end{lemma}

\begin{proof}
By \eqref{eq:standard-divisor-totient-estimates}, for every $\alpha>0$,
\[
  \frac{d(q)}{\vphi(q)}
  =
  \frac{d(q)}{q}\frac{q}{\vphi(q)}
  \lesssim_\alpha q^{-1+2\alpha}.
\]
Taking $\alpha=\varepsilon/4$ and summing over
$Q/2\le q<Q$, we obtain
\[
  \sum_{Q/2\le q<Q}\frac{d(q)}{\vphi(q)}
  \lesssim_\varepsilon Q^{\varepsilon/2}
  \le
  C_\varepsilon Q^\varepsilon.
\]
This proves \eqref{eq:arithmetic-sum}.
\end{proof}

\begin{lemma}\label{prop:Lp-major-arc-growth}
Let $1<p_0<\infty$ and $\varepsilon>0$. Then
\begin{equation}\label{eq:Lp-major-arc-growth}
  \bigl\|(\Tcal_{s,N}f)_{N\ge1}\bigr\|_{\Lpinf {p_0}\Acal}
  \le
  C_{p_0,\varepsilon}2^{\varepsilon s}
  \norm{f}_{\Lp {p_0}\Acal},
  \qquad s\ge1.
\end{equation}
The constant is independent of $s$ and $N$.
\end{lemma}

\begin{proof}
By the triangle inequality, the definition of $\Tcal_{s,N}$, and
\Cref{lem:reduced-residues},
\begin{equation*}
\begin{aligned}
  \bigl\|(\Tcal_{s,N}f)_{N\ge1}\bigr\|_{\Lpinf {p_0}\Acal}
  &\le
  \sum_{Q_s/2\le q<Q_s}
  \frac{|\mu(q)|}{\vphi(q)}
  \bigl\|(\Rcal_{q,s,N}f)_{N\ge1}\bigr\|_{\Lpinf {p_0}\Acal}\\
  &\le
  C_{p_0}
  \sum_{Q_s/2\le q<Q_s}
  \frac{d(q)}{\vphi(q)}
  \norm{f}_{\Lp {p_0}\Acal}.
\end{aligned}
\end{equation*}
Here we have used $|\mu(q)|\le1$. Applying
\Cref{lem:arithmetic-sum} with $Q=Q_s=2^{s+1}$ gives
\begin{equation*}
  \bigl\|(\Tcal_{s,N}f)_{N\ge1}\bigr\|_{\Lpinf {p_0}\Acal}
  \le
  C_{p_0,\varepsilon}Q_s^\varepsilon
  \norm{f}_{\Lp {p_0}\Acal}
  =
  C_{p_0,\varepsilon}2^{\varepsilon s}
  \norm{f}_{\Lp {p_0}\Acal}.
\end{equation*}
\end{proof}

\begin{lemma}\label{prop:L2-major-arc-gain}
There exist constants $C,\delta>0$ such that
\begin{equation}\label{eq:L2-major-arc-gain}
  \bigl\|(\Tcal_{s,N}f)_{N\ge1}\bigr\|_{\Lpinf 2\Acal}
  \le
  C2^{-\delta s}\norm{f}_{\Lp 2\Acal},
  \qquad s\ge1.
\end{equation}
\end{lemma}

\begin{proof}
Fix $s\ge1$ and set $Q:=2^{s+1}$. For $Q/2\le q<Q$, let $P_{q,s}$ be
the Fourier multiplier defined on $L_2(\mathcal A)$ by
\begin{equation*}
  \widehat{P_{q,s}f}(\xi)
  :=
  \sum_{\substack{1\le a\le q\\(a,q)=1}}
  \widetilde\eta\!\left(
    D_s\!\left(\xi-\frac{a}{q}\right)
  \right)
  \widehat f(\xi),
  \qquad
  f_{q,s}:=P_{q,s}f.
\end{equation*}
The disjointness of the enlarged major arcs and Plancherel's theorem yield
\begin{equation}\label{eq:Pqs-orthogonality}
  \sum_{Q/2\le q<Q}\norm{f_{q,s}}_{\Lp 2\Acal}^2
  \le
  \norm{f}_{\Lp 2\Acal}^2.
\end{equation}
Since $\widetilde\eta=1$ on $\supp (\eta)$,
\begin{equation}\label{eq:Rqs-localization}
  \Rcal_{q,s,N}f
  =
  \Rcal_{q,s,N}f_{q,s}.
\end{equation}
Therefore, \Cref{lem:reduced-residues},
\eqref{eq:Rqs-localization}, the Cauchy--Schwarz inequality, and
\eqref{eq:Pqs-orthogonality} imply
\begin{equation}\label{eq:Tqs-L2-preliminary}
\begin{aligned}
  \bigl\|(\Tcal_{s,N}f)_{N\ge1}\bigr\|_{\Lpinf 2\Acal}
  &\le
  C\sum_{Q/2\le q<Q}
  \frac{d(q)}{\vphi(q)}\norm{f_{q,s}}_{\Lp 2\Acal}\\
  &\le
  C\left(
    \sum_{Q/2\le q<Q}
    \left(\frac{d(q)}{\vphi(q)}\right)^2
  \right)^{1/2}
  \left(
    \sum_{Q/2\le q<Q}\norm{f_{q,s}}_{\Lp 2\Acal}^2
  \right)^{1/2}\\
  &\le
  C\left(
    \sum_{Q/2\le q<Q}
    \left(\frac{d(q)}{\vphi(q)}\right)^2
  \right)^{1/2}
  \norm{f}_{\Lp 2\Acal}.
\end{aligned}
\end{equation}
By \eqref{eq:standard-divisor-totient-estimates}, for every $\rho>0$,
\begin{equation*}
  d(q)\lesssim_\rho q^\rho,
  \qquad
  \vphi(q)^{-1}\lesssim_\rho q^{-1+\rho}
\end{equation*}
give
\begin{equation*}
  \sum_{Q/2\le q<Q}
  \left(\frac{d(q)}{\vphi(q)}\right)^2
  \lesssim_\rho Q^{-1+4\rho}.
\end{equation*}
Taking $\rho=1/16$, we obtain from
\eqref{eq:Tqs-L2-preliminary} that
\[
  \bigl\|(\Tcal_{s,N}f)_{N\ge1}\bigr\|_{\Lpinf 2\Acal}
  \le
  CQ^{-3/8}\norm{f}_{\Lp 2\Acal}.
\]
Since $Q=2^{s+1}$, this proves \eqref{eq:L2-major-arc-gain}.
\end{proof}

We can now interpolate the preceding two estimates. 
\begin{lemma}\label{lem:summation-over-major-arcs}
Let $1<p<2$. Then
\begin{equation}\label{eq:sum-over-major-arcs}
  \sum_{s\ge1}
  \bigl\|(\Tcal_{s,N}f)_{N\ge1}\bigr\|_{\Lpinf p\Acal}
  \le
  C_p\norm{f}_{\Lp p\Acal}.
\end{equation}
Consequently,
\begin{equation}\label{eq:LN-maximal-r}
  \bigl\|(L_N*f)_{N\ge1}\bigr\|_{\Lpinf p\Acal}
  \le
  C_p\norm{f}_{\Lp p\Acal}.
\end{equation}
\end{lemma}

\begin{proof}
Choose $p_0$ such that $1<p_0<p$, and let $\theta\in(0,1)$ be
determined by
\begin{equation*}
  \frac{1}{p}
  =
  \frac{1-\theta}{p_0}+\frac{\theta}{2}.
\end{equation*}
Interpolating \eqref{eq:Lp-major-arc-growth} with
\eqref{eq:L2-major-arc-gain}, using
\Cref{lem:maximal-interpolation}, gives
\begin{equation}\label{eq:interpolated-major-arc-bound}
  \bigl\|(\Tcal_{s,N}f)_{N\ge1}\bigr\|_{\Lpinf p\Acal}
  \le
  C_{p,\varepsilon}
  2^{((1-\theta)\varepsilon-\theta\delta)s}
  \norm{f}_{\Lp p\Acal}.
\end{equation}
Choose $\varepsilon>0$ sufficiently small so that
\begin{equation*}
  c_p
  :=
  \theta\delta-(1-\theta)\varepsilon
  >0.
\end{equation*}
Then \eqref{eq:interpolated-major-arc-bound} becomes
\begin{equation*}
  \bigl\|(\Tcal_{s,N}f)_{N\ge1}\bigr\|_{\Lpinf p\Acal}
  \le
  C_p2^{-c_ps}\norm{f}_{\Lp p\Acal}.
\end{equation*}
Summing over $s\ge1$ proves \eqref{eq:sum-over-major-arcs}.
Finally, combining \eqref{eq:sum-over-major-arcs}, the decomposition
\eqref{eq:LN-decomposition}, the triangle inequality in
$\Lpinf p\Acal$, and the low-frequency estimate
\eqref{eq:low-frequency-maximal}, we obtain
\eqref{eq:LN-maximal-r}.
\end{proof}

\subsection{The dyadic maximal inequality}
\label{subsec:proof-dyadic-prime-maximal}

We now prove the dyadic maximal inequality stated in
\Cref{prop:dyadic-prime-maximal}.

\begin{proof}[Proof of \Cref{prop:dyadic-prime-maximal}.]

Fix $1<p<2$, and choose $q$ such that
\begin{equation*}
  1<q<p<2.
\end{equation*}
By \Cref{lem:summation-over-major-arcs}, applied with exponent $q$,
\begin{equation*}
  \bigl\|(L_N*f)_{N\ge1}\bigr\|_{\Lpinf q\Acal}
  \le
  C_q\norm{f}_{\Lp q\Acal}.
\end{equation*}
In particular,
\begin{equation}\label{eq:LN-uniform-Lq}
  \sup_{N\ge1}\norm{L_N*f}_{\Lp q\Acal}
  \le
  C_q\norm{f}_{\Lp q\Acal}.
\end{equation}

For $k\ge1$, set
\begin{equation*}
  E_k:=K_{2^k}-L_{2^k}.
\end{equation*}
The Chebyshev estimate \eqref{eq:chebyshev-theta-upper} gives
\begin{equation*}
  \sup_{N\ge2}\norm{K_N}_{\ell_1(\Z)}<\infty.
\end{equation*}
Consequently, Young's inequality and \eqref{eq:LN-uniform-Lq} yield
\begin{equation}\label{eq:Ek-Lq-bound}
\begin{aligned}
  \norm{E_k*f}_{\Lp q\Acal}
  &\le
  \norm{K_{2^k}*f}_{\Lp q\Acal}
  +
  \norm{L_{2^k}*f}_{\Lp q\Acal}\\
  &\le
  C_q\norm{f}_{\Lp q\Acal},
  \qquad k\ge1.
\end{aligned}
\end{equation}

By the approximation estimate
\eqref{eq:approximation-error} and Plancherel's theorem imply that, for every $A>0$,
\begin{equation}\label{eq:Ek-L2-bound}
  \norm{E_k*f}_{\Lp 2\Acal}
  \le
  C_Ak^{-A}\norm{f}_{\Lp 2\Acal},
  \qquad k\ge1.
\end{equation}
Let $\vartheta\in(0,1)$ be determined by
\begin{equation*}
  \frac{1}{p}
  =
  \frac{1-\vartheta}{q}
  +
  \frac{\vartheta}{2}.
\end{equation*}
Interpolating \eqref{eq:Ek-Lq-bound} and \eqref{eq:Ek-L2-bound}, we
obtain
\begin{equation*}
  \norm{E_k*f}_{\Lp p\Acal}
  \le
  C_{p,A}k^{-A\vartheta}\norm{f}_{\Lp p\Acal}.
\end{equation*}
Choose $A>\vartheta^{-1}$. Since the sequence having $E_k*f$ in its
$k$-th coordinate and zero elsewhere has
$L_p(\Acal;\ell_\infty)$-norm equal to
$\norm{E_k*f}_{\Lp p\Acal}$, the triangle inequality in
$L_p(\Acal;\ell_\infty)$ gives
\begin{equation}\label{eq:Ek-vector-valued-bound}
\begin{aligned}
  \bigl\|(E_k*f)_{k\ge1}\bigr\|_{\Lpinf p\Acal}
  &\le
  \sum_{k\ge1}\norm{E_k*f}_{\Lp p\Acal}\\
  &\le
  C_{p,A}
  \sum_{k\ge1}k^{-A\vartheta}
  \norm{f}_{\Lp p\Acal}\\
  &\le
  C_p\norm{f}_{\Lp p\Acal}.
\end{aligned}
\end{equation}
Finally,
\[
  (K_{2^k}*f)_{k\ge1}
  =
  (L_{2^k}*f)_{k\ge1}
  +(E_k*f)_{k\ge1}.
\]
The first term is bounded by \eqref{eq:LN-maximal-r}, whereas the
second is bounded by \eqref{eq:Ek-vector-valued-bound}. This proves
\eqref{eq:dyadic-prime-maximal}.
\end{proof}

\section{Bilateral almost uniform convergence for the weighted prime averages}
\label{sec:weighted-prime-bau-convergence}

In this section, we prove the b.a.u. convergence assertion in
\Cref{thm:weighted-prime-maximal-Lp} for the weighted prime averages
\[
  B_N(x)
  :=
  \frac{1}{N}
  \sum_{q\in P_N}(\log q)\gamma^q(x),
  \qquad N\geq2.
\]
In Bourgain's scalar argument
\cite[Section~7]{BourgainMaximal}, the quantitative maximal inequality
is combined with pointwise maximal functions and exceptional sets to
obtain almost everywhere convergence by contradiction.  Such a pointwise argument is not available in the
noncommutative setting.  Our main observation is that the
noncommutative Bourgain-type quantitative maximal inequality below
already implies
\[
  \lim_{N\to\infty}
  \left\|
    \bigl(B_n(x)-B_N(x)\bigr)_{n\geq N}
  \right\|_{L_2(\Ncal;\ell_\infty)}
  =0
\]
for every $x\in \mathcal S(\mathcal N)$.  It follows that,
for some $y\in L_2(\Ncal)$,
\[
  \bigl(B_n(x)-y\bigr)_{n\geq2}
  \in L_2(\Ncal;c_0),
\]
which gives the desired b.a.u. convergence.  This $c_0$-valued conclusion is stronger than the
corresponding scalar almost everywhere convergence and also gives an
alternative proof in the commutative setting.  We first prove the
convergence result from the
quantitative maximal inequality and then prove that inequality in the
second part of the section.

Fix $0<\varepsilon<1/4$, and define
\begin{equation*}
  \mathbb D_\varepsilon
  :=
  \left\{
    \left\lfloor(1+\varepsilon)^r\right\rfloor:
    r=1,2,\ldots
  \right\}.
\end{equation*}
For
\[
  2\leq N_0<N_1<\cdots<N_{J+1},
  \qquad
  N_i\in\mathbb D_\varepsilon,
  \qquad
  N_{i+1}>2N_i,
\]
set
\[
  \mathbb D_{\varepsilon,i}
  :=
  \mathbb D_\varepsilon\cap[N_i,N_{i+1}).
\]

The key instrumental estimate of this section is the following noncommutative
Bourgain-type quantitative maximal inequality.

\begin{proposition}\label{prop:ergodic-grid-oscillation}
There exists a sequence $c_J=c_J(\varepsilon)$ satisfying
$c_J\to0$ as $J\to\infty$ such that
\begin{equation}\label{eq:ergodic-grid-oscillation}
  \frac{1}{J+1}
  \sum_{i=0}^{J}
  \left\|
    \bigl(
      B_N(x)-B_{N_i}(x)
    \bigr)_{N\in\mathbb D_{\varepsilon,i}}
  \right\|_{L_2(\Ncal;\ell_\infty)}^2
  \leq
  c_J\|x\|_{L_2(\Ncal)}^2
\end{equation}
for every $x\in \mathcal S(\mathcal N)$, uniformly over
all choices of $(N_i)_{i=0}^{J+1}$ satisfying the conditions above.
\end{proposition}
The proof of \Cref{prop:ergodic-grid-oscillation} is postponed to
\Cref{subsec:proof-quantitative-maximal}.

\subsection{From the quantitative maximal inequality to b.a.u. convergence}
\label{subsec:vanishing-oscillation}

We first use \Cref{prop:ergodic-grid-oscillation} to prove the b.a.u.
convergence of the weighted prime averages.  

\begin{lemma}\label{lem:weighted-prime-L2-Cauchy}
For every $x\in \mathcal S(\mathcal N)$,
\begin{equation*}
  \lim_{n\to\infty}
  \left\|
    \bigl(B_m(x)-B_n(x)\bigr)_{m\geq n}
  \right\|_{L_2(\Ncal;\ell_\infty)}
  =0.
\end{equation*}
\end{lemma}

\begin{proof}
By linearity, it is enough to consider $x\ge0$ with
$\|x\|_{L_\infty(\Ncal)}\le1$. Suppose, to the contrary, that the maximal norm
does not tend to zero; that is,
\[
  \left\|
    (B_m(x)-B_n(x))_{m\ge n}
  \right\|_{L_2(\Ncal;\ell_\infty)}
  \not\longrightarrow0.
\]
Then there exist $\rho>0$ and arbitrarily large integers $n$ for which
\begin{equation*}
  \left\|
    \bigl(B_m(x)-B_n(x)\bigr)_{m\ge n}
  \right\|_{L_2(\Ncal;\ell_\infty)}
  >4\rho.
\end{equation*}
Together with \eqref{eq:finite-subset-norm-identity}, one may consecutively choose a sequence $(N_i)_{i\geq0}\subset\mathbb D_\varepsilon$ with
$N_{i+1}>2N_i$ such that
\begin{equation}\label{eq:bad-full-blocks-new}
  \left\|
    \bigl(
      B_N(x)-B_{N_i}(x)
    \bigr)_{N_i\leq N<N_{i+1}}
  \right\|_{L_2(\Ncal;\ell_\infty)}
  >2\rho
\end{equation}
for every $i$.

Choose $0<\varepsilon<1/4$ sufficiently small so that
\begin{equation*}
  C\varepsilon^{1/4}\|x\|_{L_2(\Ncal)}<\rho,
\end{equation*}
where $C$ is the least constant appearing in \eqref{eq:short-variation-L2-semifinite} below.
For $N_i\leq N<N_{i+1}$, let
\[
  N'
  :=
  \max\left(
    \{N_i\}\cup
    \bigl(\mathbb D_\varepsilon\cap[2,N]\bigr)
  \right).
\]
Then $N'\in\mathbb D_{\varepsilon,i}$, and
\[
  N'\leq N<(1+\varepsilon)N'+1.
\]

For each fixed $i$, define the linear operator
\[
  \mathscr S_{\varepsilon,i}x
  :=
  \bigl(
    B_N(x)-B_{N'}(x)
  \bigr)_{N_i\leq N<N_{i+1}}.
\]
Then
\begin{equation}\label{eq:short-variation-Linfty-semifinite}
  \|\mathscr S_{\varepsilon,i}x\|_
    {L_\infty(\Ncal;\ell_\infty)}
  \leq
  C\varepsilon\|x\|_{L_\infty(\Ncal)}.
\end{equation}
Indeed, for each fixed
$0<\varepsilon<1/4$,  by \eqref{eq:PNT-pi-form}, there exists
$N_\varepsilon\geq2$ such that
\begin{equation}\label{eq:PNT-epsilon-form}
  |\vartheta(M)-M|
  \leq
  \varepsilon M,
  \qquad M\geq N_\varepsilon.
\end{equation}
For $N'\leq N<(1+\varepsilon)N'+1$, we have
\begin{align*}
  B_N(x)-B_{N'}(x)
  =
  \frac1N
  \sum_{\substack{N'<q\leq N\\q\ {\rm prime}}}
  (\log q)\gamma^q(x)-
  \frac{N-N'}{NN'}
  \sum_{\substack{q\leq N'\\q\ {\rm prime}}}
  (\log q)\gamma^q(x).
\end{align*}
Since $\gamma$ is a $*$-automorphism, taking operator norms gives
\begin{align*}
  \|B_N(x)-B_{N'}(x)\|_{L_\infty(\Ncal)}
  &\leq
  \left(
    \frac{\vartheta(N)-\vartheta(N')}{N}
    +
    \frac{N-N'}{NN'}\vartheta(N')
  \right)\|x\|_{L_\infty(\Ncal)}
  \notag\\
  &\leq
  C\varepsilon\|x\|_{L_\infty(\Ncal)},
\end{align*}
where the last inequality follows from \eqref{eq:PNT-epsilon-form},
provided that $N_i$ is sufficiently large (depending on
$\varepsilon$).
Therefore,
\begin{equation*}
  \|\mathscr S_{\varepsilon,i}x\|_
    {L_\infty(\Ncal;\ell_\infty)}
  =
  \sup_{N_i\leq N<N_{i+1}}
  \|B_N(x)-B_{N'}(x)\|_{L_\infty(\Ncal)}
  \leq
  C\varepsilon\|x\|_{L_\infty(\Ncal)}.
\end{equation*}

On the other hand, the maximal part of
\Cref{thm:weighted-prime-maximal-Lp} gives
\begin{align}
  \|\mathscr S_{\varepsilon,i}x\|_
    {L_{3/2}(\Ncal;\ell_\infty)}
  &\leq
  \left\|
    \bigl(B_N(x)\bigr)_{N_i\leq N<N_{i+1}}
  \right\|_{L_{3/2}(\Ncal;\ell_\infty)}
  \notag\\
  &\quad+
  \left\|
    \bigl(B_{N'}(x)\bigr)_{N_i\leq N<N_{i+1}}
  \right\|_{L_{3/2}(\Ncal;\ell_\infty)}
  \notag\\
  &\leq
  C\|x\|_{L_{3/2}(\Ncal)}.
  \label{eq:short-variation-Lthreehalves}
\end{align}

Interpolating
\eqref{eq:short-variation-Linfty-semifinite} and
\eqref{eq:short-variation-Lthreehalves}, 
we obtain
\begin{equation}\label{eq:short-variation-L2-semifinite}
  \left\|
    \bigl(
      B_N(x)-B_{N'}(x)
    \bigr)_{N_i\leq N<N_{i+1}}
  \right\|_{L_2(\Ncal;\ell_\infty)}
  \leq
  C\varepsilon^{1/4}\|x\|_{L_2(\Ncal)}
  <\rho.
\end{equation}

For every $N_i\leq N<N_{i+1}$,
\[
  B_N(x)-B_{N_i}(x)
  =
  \bigl(B_N(x)-B_{N'}(x)\bigr)
  +
  \bigl(B_{N'}(x)-B_{N_i}(x)\bigr).
\]
Consequently,
\begin{align*}
  &
  \left\|
    \bigl(
      B_N(x)-B_{N_i}(x)
    \bigr)_{N_i\leq N<N_{i+1}}
  \right\|_{L_2(\Ncal;\ell_\infty)}
  \\
  &\qquad<
  \rho+
  \left\|
    \bigl(
      B_N(x)-B_{N_i}(x)
    \bigr)_{N\in\mathbb D_{\varepsilon,i}}
  \right\|_{L_2(\Ncal;\ell_\infty)}.
\end{align*}
Combining this with \eqref{eq:bad-full-blocks-new} gives
\begin{equation*}
  \left\|
    \bigl(
      B_N(x)-B_{N_i}(x)
    \bigr)_{N\in\mathbb D_{\varepsilon,i}}
  \right\|_{L_2(\Ncal;\ell_\infty)}
 >\rho.
\end{equation*}
Consequently,
\begin{equation*}
  \frac1{J+1}
  \sum_{i=0}^{J}
  \left\|
    \bigl(
      B_N(x)-B_{N_i}(x)
    \bigr)_{N\in\mathbb D_{\varepsilon,i}}
  \right\|_{L_2(\Ncal;\ell_\infty)}^2
  >\rho^2.
\end{equation*}
This contradicts \eqref{eq:ergodic-grid-oscillation}, since
$c_J\to0$.
 Therefore
\[
  \lim_{n\to\infty}
  \left\|
    (B_m(x)-B_n(x))_{m\ge n}
  \right\|_{L_2(\Ncal;\ell_\infty)}
  =0.
\]

\end{proof}

\begin{proposition}\label{prop:weighted-prime-bau-bounded}
For every $x\in \mathcal S(\mathcal N)$, the sequence
$(B_N(x))_{N\ge2}$ converges b.a.u.
\end{proposition}

\begin{proof}
By \Cref{lem:weighted-prime-L2-Cauchy},
$(B_N(x))_{N\geq2}$ is Cauchy in $L_2(\Ncal)$.
Hence there exists $y\in L_2(\Ncal)$ such that
\[
  B_N(x)\longrightarrow y
  \qquad\text{in }L_2(\Ncal).
\]
We now show that the convergence is in fact $c_0$-valued. For every
$k\geq2$, the triangle inequality in
$L_2(\Ncal;\ell_\infty)$ gives
\begin{align}
  &
  \left\|
    \bigl(
      B_m(x)-y
    \bigr)_{m\geq k}
  \right\|_{L_2(\Ncal;\ell_\infty)}
  \notag\\
  &\qquad\leq
  \left\|
    \bigl(
      B_m(x)-B_k(x)
    \bigr)_{m\geq k}
  \right\|_{L_2(\Ncal;\ell_\infty)}
  +
  \left\|
    \bigl(
      B_k(x)-y
    \bigr)_{m\geq k}
  \right\|_{L_2(\Ncal;\ell_\infty)}
  \notag\\
  &\qquad\leq
  \left\|
    \bigl(
      B_m(x)-B_k(x)
    \bigr)_{m\geq k}
  \right\|_{L_2(\Ncal;\ell_\infty)}
  +
   \|B_k(x)-y\|_{L_2(\Ncal)}.
  \label{eq:BN-limit-tail-estimate}
\end{align}
Consequently, \eqref{eq:BN-limit-tail-estimate} implies
\begin{align*}
  \lim_{k\to\infty}
  \left\|
    \bigl(
      B_m(x)-y
    \bigr)_{m\geq k}
  \right\|_{L_2(\Ncal;\ell_\infty)}
  =0,
\end{align*}
since
\[
  \lim_{k\to\infty}
  \left\|
    \bigl(
      B_m(x)-B_k(x)
    \bigr)_{m\geq k}
  \right\|_{L_2(\Ncal;\ell_\infty)}
  =0
\]
and $\|B_k(x)-y\|_{L_2(\Ncal)}\to0$.

For $k\geq2$, define the finitely supported sequence
\[
  z^{(k)}_m
  :=
  \begin{cases}
    B_m(x)-y, & 2\leq m<k,\\
    0,        & m\geq k.
  \end{cases}
\]
Then
\[
  \left\|
    \bigl(
      B_m(x)-y
    \bigr)_{m\geq2}
    -
    \bigl(
      z^{(k)}_m
    \bigr)_{m\geq2}
  \right\|_{L_2(\Ncal;\ell_\infty)}
  =
  \left\|
    \bigl(
      B_m(x)-y
    \bigr)_{m\geq k}
  \right\|_{L_2(\Ncal;\ell_\infty)}
  \longrightarrow0.
\]
Since $L_2(\Ncal;c_0)$ is the closure in
$L_2(\Ncal;\ell_\infty)$ of the finitely supported sequences, it
follows that
\begin{equation*}
  \bigl(
    B_N(x)-y
  \bigr)_{N\geq2}
  \in
  L_2(\Ncal;c_0).
\end{equation*}
Finally, by \eqref{eq:c0-implies-bau},
\[
  B_N(x)-y\longrightarrow0
  \qquad\text{b.a.u.},
\]
and hence $B_N(x)\to y$ b.a.u.
\end{proof}

\begin{proof}[Completion of the proof of
\Cref{thm:weighted-prime-maximal-Lp}]
The maximal assertion was proved at the end of
\Cref{sec:weighted-prime-maximal-proof}.  Fix $1<p<\infty$.  Since
$\mathcal S(\mathcal N)$ is dense in $L_p(\Ncal)$,
\Cref{prop:weighted-prime-bau-bounded} gives b.a.u. convergence on a
dense subspace of $L_p(\Ncal)$.  Together with the maximal inequality
\begin{equation*}
  \|(B_Nz)_{N\ge2}\|_{L_p(\Ncal;\ell_\infty)}
  \le
  C_p\|z\|_{L_p(\Ncal)}\quad\forall z\in L_p(\mathcal N),
\end{equation*}
the noncommutative Banach principle extends this convergence to every
$x\in L_p(\Ncal)$.  The limit belongs to $L_p(\Ncal)$ by the same
maximal bound.  This completes the proof.
\end{proof}

\subsection{The noncommutative Bourgain-type quantitative
maximal inequality}
\label{subsec:proof-quantitative-maximal}

Our aim in this subsection is to prove
\Cref{prop:ergodic-grid-oscillation}. The proof is based on the
corresponding estimate for the discrete averages on $\Z$, from which
the ergodic estimate follows by the standard noncommutative
Calder\'on transference argument; see
\cite[Section~4.2]{ChenHongWang+arXiv2024}.

Given a reasonable  $f:\mathbb Z\rightarrow\mathcal M$, define
\[
  \mathcal B_Nf(j)
  :=
  \frac{1}{N}
  \sum_{\substack{p\leq N\\ p\ \mathrm{prime}}}
  (\log p)f(j-p),
  \qquad j\in\Z.
\]
With the convolution convention fixed in
\Cref{subsec:prime-kernels}, we have
\[
  \mathcal B_Nf=K_N*f.
\]

The required discrete estimate is the following.

\begin{proposition}\label{prop:nc-prime-lacunary-oscillation}
There exists a sequence $c_J=c_J(\varepsilon)$ satisfying
$c_J\to0$ as $J\to\infty$ such that
\begin{equation*}
  \frac{1}{J+1}
  \sum_{i=0}^{J}
  \left\|
    \bigl(
      \mathcal B_Nf-\mathcal B_{N_i}f
    \bigr)_{N\in\mathbb D_{\varepsilon,i}}
  \right\|_{L_2(\Acal;\ell_\infty)}^2
  \leq
  c_J\|f\|_{L_2(\Acal)}^2.
\end{equation*}
\end{proposition}

Assuming
\Cref{prop:nc-prime-lacunary-oscillation}, whose proof will be given
at the end of this subsection, we now prove
\Cref{prop:ergodic-grid-oscillation}.

\begin{proof}[Reduction of \Cref{prop:ergodic-grid-oscillation} to
\Cref{prop:nc-prime-lacunary-oscillation}]
For $L>N_{J+1}$, define
\[
  f_L(j)
  :=
  \gamma^{-j}(x)\1_{\{1\leq j\leq L\}},
  \qquad j\in\Z.
\]
If $N\leq N_{J+1}$ and
$N_{J+1}+1\leq j\leq L$, then
\begin{equation}\label{eq:orbit-transference-identity-new}
  \mathcal B_Nf_L(j)
  =
  \gamma^{-j}(B_N(x)).
\end{equation}
Moreover, since $\gamma$ is trace-preserving,
\[
  \|f_L\|_{L_2(\Acal)}^2
  =
  \sum_{j=1}^{L}
  \|\gamma^{-j}(x)\|_{L_2(\Ncal)}^2
  =
  L\|x\|_{L_2(\Ncal)}^2.
\]
Applying \Cref{prop:nc-prime-lacunary-oscillation} to $f_L$,
restricting the spatial variable to
\[
  N_{J+1}+1\leq j\leq L,
\]
and using \eqref{eq:orbit-transference-identity-new}, we obtain
\begin{align*}
  \frac{L-N_{J+1}}{J+1}
  \sum_{i=0}^{J}
  \left\|
    \bigl(
      B_N(x)-B_{N_i}(x)
    \bigr)_{N\in\mathbb D_{\varepsilon,i}}
  \right\|_{L_2(\Ncal;\ell_\infty)}^2
  \leq
  c_JL\|x\|_{L_2(\Ncal)}^2.
\end{align*}
 Dividing by $L-N_{J+1}$ and letting
$L\to\infty$ gives \eqref{eq:ergodic-grid-oscillation}.
\end{proof}

It remains to prove
\Cref{prop:nc-prime-lacunary-oscillation}.
We follow the general
scheme of Bourgain's quantitative maximal argument, see
\cite{BourgainMaximal,BourgainApproach}. One key intermediate estimate is \eqref{periodic-continuous-model-oscillation} appearing in Corollary \ref{lem:periodic-continuous-model-oscillation}. We need several
auxiliary results on $\mathbb R$. 
The first one is an operator-space analogue of
the Sobolev inequality.
Let $W^{1,2}\bigl(\R;L_2(\Mcal)\bigr)$ denote the space of all
strongly measurable functions
$F\colon\R\to L_2(\Mcal)$ such that both $F$ and its weak derivative
$F'$ belong to $L_2\bigl(\R;L_2(\Mcal)\bigr)$.

\begin{lemma}
\label{lem:nc-sobolev-maximal-parameter}
Let $ F\in W^{1,2}\bigl(\R;L_2(\Mcal)\bigr)$. Then
\begin{equation}\label{eq:nc-sobolev-maximal-parameter}
\begin{aligned}
  \left\|
    (F(s))_{s\in\R}
  \right\|_{L_2(\Mcal;\ell_\infty)}^2
  &\leq
  2
  \left(
    \int_\R
    \|F(s)\|_{L_2(\Mcal)}^2\,ds
  \right)^{1/2}
  \left(
    \int_\R
    \|F'(s)\|_{L_2(\Mcal)}^2\,ds
  \right)^{1/2}.
\end{aligned}
\end{equation}
\end{lemma}

\begin{proof}
Finite sums of the form
\[
  F_0(s)
  =
  \sum_{j=1}^m\phi_j(s)x_j,
  \qquad
  \phi_j\in C_c^\infty(\R),
  \quad
  x_j\in\Mcal\cap L_2(\Mcal),
\]
are dense in $W^{1,2}(\R;L_2(\Mcal))$. By approximation, it suffices to prove the estimate for such an $F_0$, which we simply denote by $F$.

For $s\in\R$, the fundamental theorem of calculus gives
\[
  F(s)^*F(s)
  =
  \int_{-\infty}^s
  \bigl(
    F'(u)^*F(u)+F(u)^*F'(u)
  \bigr)\,du.
\]
For every $\lambda>0$, the inequality
\[
  |\lambda^{1/2}X-\lambda^{-1/2}Y|^2\geq0
\]
implies
\[
  X^*Y+Y^*X
  \leq
  \lambda X^*X+\lambda^{-1}Y^*Y.
\]
Consequently,
\begin{equation}\label{eq:FstarF-majorant}
  F(s)^*F(s)
  \leq
  \lambda
  \int_\R F'(u)^*F'(u)\,du
  +
  \lambda^{-1}
  \int_\R F(u)^*F(u)\,du.
\end{equation}
Applying the same argument to $F(s)^*$ gives
\begin{equation}\label{eq:FFstar-majorant}
  F(s)F(s)^*
  \leq
  \lambda
  \int_\R F'(u)F'(u)^*\,du
  +
  \lambda^{-1}
  \int_\R F(u)F(u)^*\,du.
\end{equation}

Set
\begin{equation*}
\begin{aligned}
  A_\lambda^2
  &:=
  \lambda
  \int_\R F'(u)F'(u)^*\,du
  +
  \lambda^{-1}
  \int_\R F(u)F(u)^*\,du,\\
  B_\lambda^2
  &:=
  \lambda
  \int_\R F'(u)^*F'(u)\,du
  +
  \lambda^{-1}
  \int_\R F(u)^*F(u)\,du.
\end{aligned}
\end{equation*}
We use the following two-sided form of the Douglas factorization
lemma for positive operators $A$ and $B$; see \cite{Douglas} and
\cite[Proposition~1.3.2]{BhatiaPositive}:
\[
  xx^*\leq A^2,
  \qquad
  x^*x\leq B^2
  \quad\Longrightarrow\quad
  x=A^{1/2}yB^{1/2}
\]
for some $y$ satisfying
\[
  \|y\|_{L_\infty(\Mcal)}\leq1.
\]
Applying this implication to \eqref{eq:FstarF-majorant} and
\eqref{eq:FFstar-majorant}, we obtain contractions $Y_s\in\Mcal$
such that
\begin{equation*}
  F(s)
  =
  A_\lambda^{1/2}Y_sB_\lambda^{1/2},
  \qquad
  \sup_s\|Y_s\|_{L_\infty(\Mcal)}\leq1.
\end{equation*}
Therefore, by the factorization definition of
$L_2(\Mcal;\ell_\infty)$,
\begin{equation*}
  \|(F(s))_s\|_{L_2(\Mcal;\ell_\infty)}
  \leq
  \|A_\lambda^{1/2}\|_{L_4(\Mcal)}
  \|B_\lambda^{1/2}\|_{L_4(\Mcal)}.
\end{equation*}

Moreover,
\begin{equation*}
\begin{aligned}
  \tau(A_\lambda^2)
  &=
  \lambda
  \int_\R\|F'(u)\|_{L_2(\Mcal)}^2\,du
  +
  \lambda^{-1}
  \int_\R\|F(u)\|_{L_2(\Mcal)}^2\,du,\\
  \tau(B_\lambda^2)
  &=
  \lambda
  \int_\R\|F'(u)\|_{L_2(\Mcal)}^2\,du
  +
  \lambda^{-1}
  \int_\R\|F(u)\|_{L_2(\Mcal)}^2\,du.
\end{aligned}
\end{equation*}
Since
\[
  \|A_\lambda^{1/2}\|_{L_4(\Mcal)}^4
  =
  \tau(A_\lambda^2),
  \qquad
  \|B_\lambda^{1/2}\|_{L_4(\Mcal)}^4
  =
  \tau(B_\lambda^2),
\]
we obtain
\[
  \|(F(s))_s\|_{L_2(\Mcal;\ell_\infty)}^2
  \leq
  \lambda
  \int_\R\|F'(u)\|_{L_2(\Mcal)}^2\,du
  +
  \lambda^{-1}
  \int_\R\|F(u)\|_{L_2(\Mcal)}^2\,du.
\]
Optimizing in $\lambda$ proves the desired estimate
\eqref{eq:nc-sobolev-maximal-parameter}.
\end{proof}

Fix a smooth function $\eta$ supported in $(-1/2,1/2)$, and choose
$\eta_0\in C_c^\infty(-1/2,1/2)$ such that
\[
  \eta_0=1
  \qquad\text{on }\supp\eta.
\]
The second auxiliary result that we need is  the following estimate for Fourier multipliers on
$\R$.  Its proof is based on the preceding operator-valued Sobolev
inequality.

\begin{lemma}\label{lem:continuous-annular-maximal-multiplier}
Let $m\in C^1(\R\setminus\{0\})$ be supported in
\[
  \left\{
    u\in\R:
    \frac12\leq |u|\leq2
  \right\}
\]
and suppose that
\[
  \|m\|_{L_\infty(\R)}\leq A,
  \qquad
  \sup_{u\neq0}|u m'(u)|\leq B.
\]
For $r>0$ and $D\geq1$, let $T_{m,r,D}$ be the Fourier
multiplier on 
with symbol
\[
  \xi\longmapsto m(r\xi)\eta_0(D\xi).
\]
Then
\begin{equation}\label{eq:continuous-annular-maximal}
  \left\|
    \bigl(T_{m,r,D}h\bigr)_{r>0}
  \right\|_{
    L_2\bigl(
      L_\infty(\R)\overline{\otimes}\Mcal;
      \ell_\infty
    \bigr)
  }
  \lesssim_\eta
  (AB)^{1/2}
  \|h\|_{L_2(L_\infty(\R)\overline{\otimes}\Mcal)}.
\end{equation}
The implicit constant is independent of $D$.
\end{lemma}

\begin{proof}
For
$h\in L_2(L_\infty(\R)\overline{\otimes}\Mcal)$, put
\[
  F(s)
  :=
  T_{m,e^s,D}h,
  \qquad s\in\R.
\]
By Plancherel's theorem and the change of variables
$u=e^s|\xi|$,
\[
  \int_\R
  \|F(s)\|_{L_2(L_\infty(\R)\overline{\otimes}\Mcal)}^2\,ds
  \lesssim_\eta
  A^2
  \|h\|_{L_2(L_\infty(\R)\overline{\otimes}\Mcal)}^2.
\]
Moreover,
\[
  \widehat{F'(s)}(\xi)
  =
  e^s\xi m'(e^s\xi)
  \eta_0(D\xi)\widehat h(\xi),
\]
and hence
\[
  \int_\R
  \|F'(s)\|_{L_2(L_\infty(\R)\overline{\otimes}\Mcal)}^2\,ds
  \lesssim_\eta
  B^2
  \|h\|_{L_2(L_\infty(\R)\overline{\otimes}\Mcal)}^2.
\]
Applying \Cref{lem:nc-sobolev-maximal-parameter} with
$L_\infty(\R)\overline{\otimes}\Mcal$ in place of $\Mcal$ proves
\eqref{eq:continuous-annular-maximal}.
\end{proof}

Let
\[
  \widehat k(\xi)
  :=
  \int_0^1e^{-2\pi i u\xi}\,du.
\]
For $t>0$ and $D\geq1$, let $S_{t,D}$ and
$T_{\eta,D}$ be the Fourier multipliers on
$L_2(L_\infty(\R)\overline{\otimes}\Mcal)$ with symbols
\[
  \xi\longmapsto
  \widehat k(t\xi)\eta_0(D\xi)
  \qquad\text{and}\qquad
  \xi\longmapsto\eta(D\xi),
\]
respectively.  Since $\eta_0=1$ on $\supp\eta$, the composition
$S_{t,D}T_{\eta,D}$ has symbol
\[
  \xi\longmapsto
  \widehat k(t\xi)\eta(D\xi).
\]

We shall use
\begin{equation}\label{eq:model-multiplier-elementary-estimates}
\begin{aligned}
  |\widehat k(u)|
  &\lesssim |u|^{-1},
  &
  |u\widehat k'(u)|
  &\lesssim1,
  && |u|\geq\frac12,\\
  |\widehat k(u)-1|
  &\lesssim |u|,
  &
  |u\widehat k'(u)|
  &\lesssim |u|,
  && |u|\leq2.
\end{aligned}
\end{equation}
Moreover, the noncommutative Hardy--Littlewood maximal inequality
\cite{Mei+MAMS2007} gives
\begin{equation}\label{eq:continuous-model-global-maximal}
  \left\|
    \bigl(S_{t,D}h\bigr)_{t>0}
  \right\|_{
    L_2\bigl(
      L_\infty(\R)\overline{\otimes}\Mcal;
      \ell_\infty
    \bigr)
  }
  \lesssim_\eta
  \|h\|_{L_2(L_\infty(\R)\overline{\otimes}\Mcal)},
\end{equation}
uniformly in $D$.

The key auxiliary result that we need is the following quantitative maximal inequality on $\mathbb R$.
\begin{lemma}\label{lem:continuous-model-oscillation}
Let $\theta_0\in\R$ and $D\geq1$, and let
\[
  1<N_0<N_1<\cdots<N_{J+1},
  \qquad
  N_{i+1}>2N_i.
\]
For $t>0$, define $S_{\theta_0,t,D}$ on
$L_2(L_\infty(\R)\overline{\otimes}\Mcal)$ by
\[
  \widehat{S_{\theta_0,t,D}f}(\xi)
  :=
  \widehat k\bigl(t(\xi-\theta_0)\bigr)
  \eta\bigl(D(\xi-\theta_0)\bigr)
  \widehat f(\xi).
\]
Then
\begin{equation}\label{eq:continuous-model-oscillation}
  \sum_{i=0}^{J}
  \left\|
    \bigl(
      S_{\theta_0,t,D}f
      -
      S_{\theta_0,N_i,D}f
    \bigr)_{N_i\leq t<N_{i+1}}
  \right\|_{
    L_2\bigl(
      L_\infty(\R)\overline{\otimes}\Mcal;
      \ell_\infty
    \bigr)
  }^2
  \lesssim_\eta
  \|f\|_{L_2(L_\infty(\R)\overline{\otimes}\Mcal)}^2,
\end{equation}
uniformly in $\theta_0$, $D$, $J$, and $(N_i)$.
\end{lemma}

\begin{proof}
By modulation, it suffices to consider $\theta_0=0$.
Put
\[
  g:=T_{\eta,D}f.
\]
Then
\[
  S_{0,t,D}f
  =
  S_{t,D}g.
\]
Extend $(N_i)_{i=0}^{J+1}$ to a two-sided sequence
$(N_i)_{i\in\Z}\subset(0,\infty)$, without changing its prescribed
terms, such that $N_{i+1}>2N_i$ for every $i\in\Z$.

For $\ell\in\Z$, define
\[
  E_\ell
  :=
  \left\{
    \xi\in\R:
    N_{\ell+1}^{-1}<|\xi|\leq N_\ell^{-1}
  \right\},
  \qquad
  \widehat{g_\ell}
  :=
  \1_{E_\ell}\widehat g.
\]
The sets $E_\ell$ are pairwise disjoint and cover
$\R\setminus\{0\}$.  Therefore,
\begin{equation}\label{eq:continuous-frequency-orthogonality}
\begin{aligned}
  g
  &=
  \sum_{\ell\in\Z}g_\ell
  \quad\text{in }
  L_2(L_\infty(\R)\overline{\otimes}\Mcal),
  \\
  \sum_{\ell\in\Z}
  \|g_\ell\|_{L_2(L_\infty(\R)\overline{\otimes}\Mcal)}^2
  &=
  \|g\|_{L_2(L_\infty(\R)\overline{\otimes}\Mcal)}^2
  \lesssim_\eta
  \|f\|_{L_2(L_\infty(\R)\overline{\otimes}\Mcal)}^2.
\end{aligned}
\end{equation}

For $i,\ell\in\Z$, set
\[
  \Omega_i(g_\ell)
  :=
  \left\|
    \bigl(
      S_{t,D}g_\ell
      -
      S_{N_i,D}g_\ell
    \bigr)_{N_i\leq t<N_{i+1}}
  \right\|_{
    L_2\bigl(
      L_\infty(\R)\overline{\otimes}\Mcal;
      \ell_\infty
    \bigr)
  }.
\]
We claim that
\begin{equation}\label{eq:continuous-model-block-off-diagonal}
  \Omega_i(g_\ell)
  \lesssim_\eta
  2^{-|i-\ell|/2}
  \|g_\ell\|_{L_2(L_\infty(\R)\overline{\otimes}\Mcal)}.
\end{equation}
Assuming this claim, the triangle inequality gives
\begin{align*}
  &\left\|
    \bigl(
      S_{t,D}g
      -
      S_{N_i,D}g
    \bigr)_{N_i\leq t<N_{i+1}}
  \right\|_{
    L_2\bigl(
      L_\infty(\R)\overline{\otimes}\Mcal;
      \ell_\infty
    \bigr)
  }
  \\
  &\qquad\leq
  \sum_{\ell\in\Z}\Omega_i(g_\ell)
  \lesssim_\eta
  \sum_{\ell\in\Z}
  2^{-|i-\ell|/2}
  \|g_\ell\|_{L_2(L_\infty(\R)\overline{\otimes}\Mcal)}.
\end{align*}
Young's inequality and
\eqref{eq:continuous-frequency-orthogonality} now imply
\eqref{eq:continuous-model-oscillation}.

It remains to prove
\eqref{eq:continuous-model-block-off-diagonal}.  If
$|i-\ell|\leq1$, the triangle inequality and
\eqref{eq:continuous-model-global-maximal} give
\[
  \Omega_i(g_\ell)
  \lesssim_\eta
  \|g_\ell\|_{L_2(L_\infty(\R)\overline{\otimes}\Mcal)}.
\]

Suppose next that $\ell\leq i-2$.  Put
\[
  R:=\frac{N_i}{N_{\ell+1}}
\]
and write
\[
  t=N_ir,
  \qquad r\geq1.
\]
The lacunarity of $(N_i)$ gives
\[
  R\geq2^{i-\ell-1}.
\]
Choose $\rho\in C_c^\infty(\R)$ such that
\[
  \supp\rho
  \subset
  \left\{
    u\in\R:
    \frac12\leq|u|\leq2
  \right\},
  \qquad
  \sum_{j\in\Z}\rho(2^{-j}u)=1
  \quad(u\neq0),
\]
and choose $\psi_+\in C^\infty(\R)$ satisfying
\[
  \psi_+(u)=0\quad(|u|\leq1/2),
  \qquad
  \psi_+(u)=1\quad(|u|\geq1).
\]
For $j\geq-1$, define
\[
  m_{R,j}(u)
  :=
  \widehat k(Ru)\psi_+(u)\rho(2^{-j}u),
  \qquad
  \widetilde m_{R,j}(u)
  :=
  m_{R,j}(2^ju).
\]
Since
\[
  \supp\widehat{g_\ell}
  \subset
  \left\{
    |\xi|>N_{\ell+1}^{-1}
  \right\},
\]
we have, on $\supp\widehat{g_\ell}$,
\[
  \widehat k(t\xi)\eta_0(D\xi)
  =
  \sum_{j\geq-1}
  \widetilde m_{R,j}
  \left(
    2^{-j}rN_{\ell+1}\xi
  \right)
  \eta_0(D\xi).
\]
Moreover, \eqref{eq:model-multiplier-elementary-estimates} gives
\[
  \|\widetilde m_{R,j}\|_{L_\infty(\R)}
  \lesssim
  (R2^j)^{-1},
  \qquad
  \sup_{u\neq0}
  |u\widetilde m_{R,j}'(u)|
  \lesssim1.
\]
Applying \Cref{lem:continuous-annular-maximal-multiplier} and summing
over $j\geq-1$, we obtain
\[
  \left\|
    \bigl(S_{t,D}g_\ell\bigr)_{t\geq N_i}
  \right\|_{
    L_2\bigl(
      L_\infty(\R)\overline{\otimes}\Mcal;
      \ell_\infty
    \bigr)
  }
  \lesssim_\eta
  R^{-1/2}
  \|g_\ell\|_{L_2(L_\infty(\R)\overline{\otimes}\Mcal)}.
\]
Consequently,
\[
  \Omega_i(g_\ell)
  \lesssim_\eta
  2^{-(i-\ell-1)/2}
  \|g_\ell\|_{L_2(L_\infty(\R)\overline{\otimes}\Mcal)}.
\]

Finally, suppose that $\ell\geq i+2$.  Put
\[
  R:=\frac{N_{i+1}}{N_\ell}
\]
and write
\[
  t=N_{i+1}r,
  \qquad 0<r\leq1.
\]
The lacunarity of $(N_i)$ gives
\[
  R\leq2^{-(\ell-i-1)}.
\]
Choose $\psi_-\in C^\infty(\R)$ satisfying
\[
  \psi_-(u)=1\quad(|u|\leq1),
  \qquad
  \psi_-(u)=0\quad(|u|\geq2).
\]
For $j\leq1$, define
\[
  n_{R,j}(u)
  :=
  \bigl(\widehat k(Ru)-1\bigr)
  \psi_-(u)\rho(2^{-j}u),
  \qquad
  \widetilde n_{R,j}(u)
  :=
  n_{R,j}(2^ju).
\]
Since
\[
  \supp\widehat{g_\ell}
  \subset
  \left\{
    |\xi|\leq N_\ell^{-1}
  \right\},
\]
we have, on $\supp\widehat{g_\ell}$,
\[
  \bigl(\widehat k(t\xi)-1\bigr)\eta_0(D\xi)
  =
  \sum_{j\leq1}
  \widetilde n_{R,j}
  \left(
    2^{-j}rN_\ell\xi
  \right)
  \eta_0(D\xi).
\]
By \eqref{eq:model-multiplier-elementary-estimates},
\[
  \|\widetilde n_{R,j}\|_{L_\infty(\R)}
  +
  \sup_{u\neq0}|u\widetilde n_{R,j}'(u)|
  \lesssim
  R2^j.
\]
It follows from
\Cref{lem:continuous-annular-maximal-multiplier} that
\[
  \left\|
    \bigl(
      (S_{t,D}-I)g_\ell
    \bigr)_{0<t\leq N_{i+1}}
  \right\|_{
    L_2\bigl(
      L_\infty(\R)\overline{\otimes}\Mcal;
      \ell_\infty
    \bigr)
  }
  \lesssim_\eta
  R
  \|g_\ell\|_{L_2(L_\infty(\R)\overline{\otimes}\Mcal)}.
\]
Here we used the fact that $\eta_0(D\xi)=1$ on
$\supp\widehat{g_\ell}$.  Since
\[
  S_{t,D}g_\ell-S_{N_i,D}g_\ell
  =
  (S_{t,D}-I)g_\ell
  -
  (S_{N_i,D}-I)g_\ell,
\]
we obtain
\[
  \Omega_i(g_\ell)
  \lesssim_\eta
  2^{-(\ell-i-1)/2}
  \|g_\ell\|_{L_2(L_\infty(\R)\overline{\otimes}\Mcal)}.
\]
This proves \eqref{eq:continuous-model-block-off-diagonal} and
completes the proof.
\end{proof}

We now pass to Fourier multipliers on $\Z$. Recalling the discrete Fourier multiplier associated with a symbol supported in $(-\frac12,\frac12)$ that has been introduced in Subsection \ref{subsec:approximate-kernel-maximal}, we have 
$(S_{\theta_0,t,D})_{\mathrm{dis}}$ on $L_2(\Acal)$ with symbol 
$\bigl(
    \widehat k(t\,\cdot)\eta(D\,\cdot)
  \bigr)_{\mathrm{per}}(\cdot-\theta_0)$
for $\theta_0\in\T$, $t>0$, and $D\geq1$.

The corresponding discrete estimate is now an immediate consequence of the sampling
principle.

\begin{corollary}\label{lem:periodic-continuous-model-oscillation}
Let $\theta_0\in\T$ and $D\geq1$ and
\[
  2\leq N_0<N_1<\cdots<N_{J+1},
  \qquad
  N_{i+1}>2N_i.
\]
Then
\begin{equation}\label{periodic-continuous-model-oscillation}
  \sum_{i=0}^{J}
  \left\|
    \bigl(
      (S_{\theta_0,t,D})_{\mathrm{dis}}f
      -
      (S_{\theta_0,N_i,D})_{\mathrm{dis}}f
    \bigr)_{N_i\leq t<N_{i+1}}
  \right\|_{L_2(\Acal;\ell_\infty)}^2
  \lesssim_\eta
  \|f\|_{L_2(\Acal)}^2,
\end{equation}
uniformly in $\theta_0$, $D$, $J$, and $(N_i)$.
\end{corollary}

\begin{proof}
For $0\leq i\leq J$ and $N_i\leq t<N_{i+1}$, set
\[
  \phi_{i,t}(\xi)
  :=
  \bigl(
    \widehat k(t\xi)
    -
    \widehat k(N_i\xi)
  \bigr)
  \eta(D\xi).
\]
These functions are supported in $(-1/2,1/2)$.  By
\Cref{lem:continuous-model-oscillation}, the hypothesis of
\Cref{lem:sampling} holds with
\[
  p=r=2,
  \qquad
  I=\{0,\ldots,J\},
  \qquad
  \Lambda_i=[N_i,N_{i+1}).
\]
Applying \Cref{lem:sampling} with $q=1$ gives the desired estimate \eqref{periodic-continuous-model-oscillation} uniformly in $\theta_0$, $D$, $J$, and $(N_i)$.
\end{proof}

We now prove
\Cref{prop:nc-prime-lacunary-oscillation}.
\begin{proof}[Proof of
\Cref{prop:nc-prime-lacunary-oscillation}]
Let $L_N$ be the approximate prime kernel constructed in
\Cref{sec:weighted-prime-maximal-proof}, so that on the geometric grid
\begin{equation}\label{eq:grid-approximation-error}
  \|\widehat K_N-\widehat L_N\|_{L_\infty(\T)}
  \lesssim_{A,\varepsilon}
  r^{-A},
  \qquad
  N=\lfloor(1+\varepsilon)^r\rfloor.
\end{equation}
For $N\geq2$, define
\begin{equation*}
\begin{aligned}
  \mathcal E_Nf
  &:=
  f*(K_N-L_N),\\
  \mathcal R_N^{(s_0)}f
  &:=
  f*(L_N-M_N^{(s_0)}),\\
  \mathcal M_N^{(s_0)}f
  &:=
  f*M_N^{(s_0)},
\end{aligned}
\end{equation*}
where
\[
  M_N^{(s_0)}
  :=
  L_N^{(0)}
  +
  \sum_{1\leq s<s_0}\Tcal_{s,N}.
\]
Thus,
\begin{equation}\label{eq:BN-three-part-decomposition}
  \mathcal B_Nf
  =
  \mathcal E_Nf
  +
  \mathcal R_N^{(s_0)}f
  +
  \mathcal M_N^{(s_0)}f.
\end{equation}

For $0\leq i\leq J$, set
\[
\begin{aligned}
  e_i
  &:=
  \left\|
    \bigl(
      \mathcal E_Nf-\mathcal E_{N_i}f
    \bigr)_{N\in\mathbb D_{\varepsilon,i}}
  \right\|_{L_2(\Acal;\ell_\infty)},\\
  r_i
  &:=
  \left\|
    \bigl(
      \mathcal R_N^{(s_0)}f
      -
      \mathcal R_{N_i}^{(s_0)}f
    \bigr)_{N\in\mathbb D_{\varepsilon,i}}
  \right\|_{L_2(\Acal;\ell_\infty)},\\
  m_i
  &:=
  \left\|
    \bigl(
      \mathcal M_N^{(s_0)}f
      -
      \mathcal M_{N_i}^{(s_0)}f
    \bigr)_{N\in\mathbb D_{\varepsilon,i}}
  \right\|_{L_2(\Acal;\ell_\infty)}.
\end{aligned}
\]
By \eqref{eq:BN-three-part-decomposition} and the triangle inequality,
\begin{equation}\label{eq:outer-Minkowski-revised}
\begin{aligned}
  &\left(
    \sum_{i=0}^{J}
    \left\|
      \bigl(
        \mathcal B_Nf-\mathcal B_{N_i}f
      \bigr)_{N\in\mathbb D_{\varepsilon,i}}
    \right\|_{L_2(\Acal;\ell_\infty)}^2
  \right)^{1/2}\\
  &\qquad\leq
  \left(\sum_{i=0}^{J}e_i^2\right)^{1/2}
  +
  \left(\sum_{i=0}^{J}r_i^2\right)^{1/2}
  +
  \left(\sum_{i=0}^{J}m_i^2\right)^{1/2}.
\end{aligned}
\end{equation}

We estimate the three terms separately.

\medskip
\noindent
\emph{Step 1.}
For every $i$,
\[
\begin{aligned}
  e_i
  &\leq
  \left\|
    (\mathcal E_Nf)_{N\in\mathbb D_{\varepsilon,i}}
  \right\|_{L_2(\Acal;\ell_\infty)}
  +
  \|\mathcal E_{N_i}f\|_{L_2(\Acal)}\\
  &\leq
  \sum_{N\in\mathbb D_{\varepsilon,i}}
  \|\mathcal E_Nf\|_{L_2(\Acal)}
  +
  \|\mathcal E_{N_i}f\|_{L_2(\Acal)}.
\end{aligned}
\]
Hence
\[
\begin{aligned}
  \left(\sum_{i=0}^{J}e_i^2\right)^{1/2}
  &\leq
  \sum_{N\in\mathbb D_\varepsilon}
  \|\mathcal E_Nf\|_{L_2(\Acal)}
  +
  \sum_{i=0}^{J}
  \|\mathcal E_{N_i}f\|_{L_2(\Acal)}.
\end{aligned}
\]
By Plancherel's theorem and
\eqref{eq:grid-approximation-error},
\[
  \|\mathcal E_Nf\|_{L_2(\Acal)}
  \leq
  C_A(\log N)^{-A}\|f\|_{L_2(\Acal)}.
\]
Choosing $A>1$ and using the geometric growth of
$\mathbb D_\varepsilon$, we obtain
\begin{equation}\label{eq:error-square-sum-final-revised}
  \left(\sum_{i=0}^{J}e_i^2\right)^{1/2}
  \lesssim_{A,\varepsilon}
  \|f\|_{L_2(\Acal)}.
\end{equation}

\medskip
\noindent
\emph{Step 2.}
Since
\[
  \mathcal R_N^{(s_0)}f
  =
  \sum_{s\geq s_0}\Tcal_{s,N}f,
\]
\Cref{prop:L2-major-arc-gain} gives
\[
\begin{aligned}
  r_i
  &\leq
  \left\|
    (\mathcal R_N^{(s_0)}f)_{N\in\mathbb D_{\varepsilon,i}}
  \right\|_{L_2(\Acal;\ell_\infty)}
  +
  \|\mathcal R_{N_i}^{(s_0)}f\|_{L_2(\Acal)}\\
  &\lesssim
  2^{-\delta s_0}\|f\|_{L_2(\Acal)}.
\end{aligned}
\]
Therefore,
\begin{equation}\label{eq:tail-square-sum-final-revised}
  \left(\sum_{i=0}^{J}r_i^2\right)^{1/2}
  \lesssim
  (J+1)^{1/2}2^{-\delta s_0}
  \|f\|_{L_2(\Acal)}.
\end{equation}

\medskip
\noindent
\emph{Step 3.}
Recall that
\begin{equation}\label{eq:finite-major-arc-model-explicit}
\begin{aligned}
  M_N^{(s_0)}
  &=
  L_N^{(0)}
  +
  \sum_{1\leq s<s_0}
  \sum_{Q_s/2\leq q<Q_s}
  \frac{\mu(q)}{\vphi(q)}
  \sum_{\substack{1\leq a\leq q\\(a,q)=1}}
  \Rcal_{a,q,s,N},
\end{aligned}
\end{equation}
where
\[
  \widehat{\Rcal_{a,q,s,N}f}(\xi)
  =
  \nu_N\left(\xi-\frac{a}{q}\right)
  \eta\left(
    D_s\left(\xi-\frac{a}{q}\right)
  \right)
  \widehat f(\xi).
\]
For $0\leq i\leq J$, set
\[
  \Omega_{i,a,q,s}(f)
  :=
  \left\|
    \bigl(
      \Rcal_{a,q,s,N}f
      -
      \Rcal_{a,q,s,N_i}f
    \bigr)_{N\in\mathbb D_{\varepsilon,i}}
  \right\|_{L_2(\Acal;\ell_\infty)}.
\]
We claim that
\begin{equation}\label{eq:q-block-oscillation-estimate}
  \sum_{i=0}^{J}
  \Omega_{i,a,q,s}(f)^2
  \lesssim_\varepsilon
  \|f\|_{L_2(\Acal)}^2,
\end{equation}
uniformly in $a$, $q$, and $s$. The corresponding estimate for the
low-frequency term is
\begin{equation}\label{eq:low-frequency-block-oscillation}
  \sum_{i=0}^{J}
  \left\|
    \bigl(
      L_N^{(0)}f-L_{N_i}^{(0)}f
    \bigr)_{N\in\mathbb D_{\varepsilon,i}}
  \right\|_{L_2(\Acal;\ell_\infty)}^2
  \lesssim_\varepsilon
  \|f\|_{L_2(\Acal)}^2.
\end{equation}
Assuming these estimates, Minkowski's inequality and
\eqref{eq:finite-major-arc-model-explicit} give
\begin{equation}\label{eq:finite-model-oscillation-explicit}
\begin{aligned}
  \left(\sum_{i=0}^{J}m_i^2\right)^{1/2}
  &\leq
  \left(
    \sum_{i=0}^{J}
    \left\|
      \bigl(
        L_N^{(0)}f-L_{N_i}^{(0)}f
      \bigr)_{N\in\mathbb D_{\varepsilon,i}}
    \right\|_{L_2(\Acal;\ell_\infty)}^2
  \right)^{1/2}\\
  &\quad+
  \sum_{1\leq s<s_0}
  \sum_{Q_s/2\leq q<Q_s}
  \frac{|\mu(q)|}{\vphi(q)}
  \sum_{\substack{1\leq a\leq q\\(a,q)=1}}
  \left(
    \sum_{i=0}^{J}
    \Omega_{i,a,q,s}(f)^2
  \right)^{1/2}\\
  &\lesssim_\varepsilon
  4^{s_0}\|f\|_{L_2(\Acal)}.
\end{aligned}
\end{equation}

It remains to prove
\eqref{eq:q-block-oscillation-estimate} and
\eqref{eq:low-frequency-block-oscillation}. Define
\[
  \mathcal D_{a,q,s,N}
  :=
  \Rcal_{a,q,s,N}
  -
  \widetilde{\Rcal}_{a,q,s,N},
\]
where
\[
  \widehat{\widetilde{\Rcal}_{a,q,s,N}f}(\xi)
  :=
  \widehat k\left(
    N\left(\xi-\frac aq\right)
  \right)
  \eta\left(
    D_s\left(\xi-\frac aq\right)
  \right)
  \widehat f(\xi).
\]
For $|\theta|\leq1/2$, the identity
\[
  \nu_N(\theta)
  =
  \widehat k(N\theta)
  e^{-2\pi i\theta}
  \frac{2\pi i\theta}{1-e^{-2\pi i\theta}}
\]
implies
\begin{equation*}
  \left|
    \nu_N(\theta)-\widehat k(N\theta)
  \right|
  \lesssim
  \min\left\{
    |\theta|,\frac1N
  \right\}.
\end{equation*}
Since $\eta(D_s\theta)$ is supported where
$|\theta|\leq(2D_s)^{-1}$, Plancherel's theorem gives
\begin{equation}\label{eq:discrete-continuous-operator-error}
  \|\mathcal D_{a,q,s,N}f\|_{L_2(\Acal)}
  \lesssim_\eta
  \min\left\{
    \frac1{D_s},\frac1N
  \right\}
  \|f\|_{L_2(\Acal)}.
\end{equation}
For $0\leq i\leq J$, define
\[
\begin{aligned}
  \widetilde\Omega_{i,a,q,s}(f)
  &:=
  \left\|
    \bigl(
      \widetilde{\Rcal}_{a,q,s,N}f
      -
      \widetilde{\Rcal}_{a,q,s,N_i}f
    \bigr)_{N\in\mathbb D_{\varepsilon,i}}
  \right\|_{L_2(\Acal;\ell_\infty)},\\
  d_{i,a,q,s}(f)
  &:=
  \left\|
    \bigl(
      \mathcal D_{a,q,s,N}f
      -
      \mathcal D_{a,q,s,N_i}f
    \bigr)_{N\in\mathbb D_{\varepsilon,i}}
  \right\|_{L_2(\Acal;\ell_\infty)}.
\end{aligned}
\]
By the triangle inequality,
\[
  \Omega_{i,a,q,s}(f)
  \leq
  \widetilde\Omega_{i,a,q,s}(f)
  +
  d_{i,a,q,s}(f).
\]
Moreover,
\[
\begin{aligned}
  d_{i,a,q,s}(f)
  &\leq
  \sum_{N\in\mathbb D_{\varepsilon,i}}
  \|\mathcal D_{a,q,s,N}f\|_{L_2(\Acal)}
  +
  \|\mathcal D_{a,q,s,N_i}f\|_{L_2(\Acal)}.
\end{aligned}
\]
Therefore, by
\eqref{eq:discrete-continuous-operator-error},
\[
\begin{aligned}
  \left(
    \sum_{i=0}^{J}
    d_{i,a,q,s}(f)^2
  \right)^{1/2}
  &\leq
  \sum_{N\in\mathbb D_\varepsilon}
  \|\mathcal D_{a,q,s,N}f\|_{L_2(\Acal)}
  +
  \sum_{i=0}^{J}
  \|\mathcal D_{a,q,s,N_i}f\|_{L_2(\Acal)}\\
  &\lesssim_\varepsilon
  \frac{1+\log(2D_s)}{D_s}
  \|f\|_{L_2(\Acal)}\\
  &\lesssim_\varepsilon
  \|f\|_{L_2(\Acal)}.
\end{aligned}
\]
Here we used the geometric growth of
$\mathbb D_\varepsilon$, the lacunarity of $(N_i)$, and the fact that
$D_s\geq1$.

On the other hand, note that $\widetilde{\Rcal}_{a,q,s,N}=(S_{a/q,N,D_s})_{\mathrm{dis}}$, and then
\Cref{lem:periodic-continuous-model-oscillation}, applied with
\[
  \theta_0=\frac aq,
  \qquad
  D=D_s,
\]
gives
\[
  \sum_{i=0}^{J}
  \widetilde\Omega_{i,a,q,s}(f)^2
  \lesssim_\eta
  \|f\|_{L_2(\Acal)}^2.
\]
Minkowski's inequality now yields
\[
\begin{aligned}
  \left(
    \sum_{i=0}^{J}
    \Omega_{i,a,q,s}(f)^2
  \right)^{1/2}
  &\leq
  \left(
    \sum_{i=0}^{J}
    \widetilde\Omega_{i,a,q,s}(f)^2
  \right)^{1/2}
  +
  \left(
    \sum_{i=0}^{J}
    d_{i,a,q,s}(f)^2
  \right)^{1/2}\\
  &\lesssim_\varepsilon
  \|f\|_{L_2(\Acal)}.
\end{aligned}
\]
This proves \eqref{eq:q-block-oscillation-estimate}. The same
argument, with $\theta_0=0$ and $D=1$, proves
\eqref{eq:low-frequency-block-oscillation}, and hence completes the
proof of \eqref{eq:finite-model-oscillation-explicit}.

\medskip
\noindent
\emph{Step 4.}
Substituting
\eqref{eq:error-square-sum-final-revised},
\eqref{eq:tail-square-sum-final-revised}, and
\eqref{eq:finite-model-oscillation-explicit} into
\eqref{eq:outer-Minkowski-revised}, we obtain
\begin{equation}\label{eq:bourgain-combined-oscillation-revised}
\begin{aligned}
  &\left(
    \sum_{i=0}^{J}
    \left\|
      \bigl(
        \mathcal B_Nf-\mathcal B_{N_i}f
      \bigr)_{N\in\mathbb D_{\varepsilon,i}}
    \right\|_{L_2(\Acal;\ell_\infty)}^2
  \right)^{1/2}\\
  &\qquad\lesssim_\varepsilon
  \left(
    1
    +4^{s_0}
    +(J+1)^{1/2}2^{-\delta s_0}
  \right)
  \|f\|_{L_2(\Acal)}.
\end{aligned}
\end{equation}
Choose
\begin{equation}\label{eq:s0-choice-revised}
  s_0
  :=
  \max\left\{
    1,
    \left\lfloor
      \kappa\log_2(J+2)
    \right\rfloor
  \right\},
  \qquad
  0<\kappa<\frac14.
\end{equation}
After squaring
\eqref{eq:bourgain-combined-oscillation-revised} and dividing by
$J+1$, we may take
\begin{equation*}
  c_J
  \lesssim_\varepsilon
  \frac1{J+1}
  +
  \frac{16^{s_0}}{J+1}
  +
  2^{-2\delta s_0}.
\end{equation*}
By \eqref{eq:s0-choice-revised},
\[
  c_J
  \lesssim_\varepsilon
  \frac1{J+1}
  +(J+2)^{4\kappa-1}
  +(J+2)^{-2\delta\kappa},
\]
which tends to zero because $\kappa<1/4$ and $\delta>0$. This proves
\Cref{prop:nc-prime-lacunary-oscillation}. 
\end{proof}
\begin{remark}\label{rem:chw-pointwise-convergence}
\normalfont
The argument developed in this section also applies to the polynomial
averages
\[
  A_N^{(d)}(x)
  :=
  \frac1N\sum_{k=1}^N\gamma^{k^d}(x)
\]
studied in \cite{ChenHongWang+arXiv2024}. Indeed, the required continuous oscillation estimate follows by
repeating the proof of \Cref{lem:continuous-model-oscillation} with
$k$ replaced by
\[
  k_d(u)
  :=
  \frac1d u^{1/d-1}\1_{[0,1]}(u).
\]
This is possible because
\[
  |\widehat{k_d}(u)|
  \lesssim_d |u|^{-1/d},
  \qquad
  |u\widehat{k_d}'(u)|
  \lesssim_d 1
\]
for large $|u|$, while
\[
  |\widehat{k_d}(u)-1|
  +
  |u\widehat{k_d}'(u)|
  \lesssim_d |u|
\]
near the origin.
With this estimate, we can follow the proof of
\cite[Lemma~4.4]{ChenHongWang+arXiv2024} with only minor
modifications. Choose the parameters $\rho$ and $\rho'$ appearing in
\cite[Lemma~2.4]{ChenHongWang+arXiv2024} so that
$1<\rho'<\rho<2$. Then the approximation estimate in that lemma,
together with \cite[(4.16)--(4.18)]{ChenHongWang+arXiv2024}, gives the
analogue of \Cref{prop:ergodic-grid-oscillation}.

Finally, if
$N'\leq N<(1+\varepsilon)N'+1$, then
\[
  \|A_N^{(d)}(x)-A_{N'}^{(d)}(x)\|_\infty
  \lesssim
  \varepsilon\|x\|_\infty.
\]
Combining this estimate with the maximal inequality in
\cite[Theorem~1.1]{ChenHongWang+arXiv2024} and repeating the
$L_2(\Ncal;c_0)$ argument above gives b.a.u. convergence for
\[
  p>\frac{1+\sqrt5}{2}.
\]
Thus this also solves the pointwise convergence problem left open in
\cite{ChenHongWang+arXiv2024}.
\end{remark}

\section{From weighted to unweighted prime averages}
\label{sec:proof-main-theorem}

We now complete the proof of \Cref{thm:prime-bau-convergence}. The
maximal inequality and b.a.u. convergence for the logarithmically
weighted averages $(B_N)$ have been established in
\Cref{thm:weighted-prime-maximal-Lp}. It remains to pass from $(B_N)$
to the original unweighted averages $(A_N)$. The argument follows the
same general idea as in the classical setting; see
\cite{Wierdl}, \cite[Section~5]{MTZK}, and
\cite[Appendix~A]{EisnerLin}. For completeness, we give the details of
the argument in the noncommutative setting.

\subsection{The unweighted maximal inequality}
\label{subsec:unweighted-maximal}
We first deduce the maximal inequality for the unweighted prime averages from the logarithmically weighted maximal inequality in \Cref{thm:weighted-prime-maximal-Lp}. 
Let $x\in L_r(\Ncal)_+$.
Recall that
\[
  A_N(x)
  =
  \frac{1}{|P_N|}\sum_{p\le N}\gamma^p(x)
\]
and
\[
  B_N(x)
  =
  \frac{1}{N}\sum_{p\le N}(\log p)\gamma^p(x),
  \qquad N\ge2.
\]
For $n\ge2$, set
\begin{equation*}
  T_n(x)
  :=
  \sum_{p\le n}(\log p)\gamma^p(x)
  =
  nB_n(x).
\end{equation*}
Applying the discrete Abel summation formula to
\[
  a_m
  :=
  \1_{\{m\ \mathrm{prime}\}}(\log m)\gamma^m(x)
  \qquad\text{and}\qquad
  b_m
  :=
  \frac{1}{\log m},
  \qquad m\ge2,
\]
we obtain
\begin{equation}\label{eq:abel-unweighted-from-weighted}
\begin{aligned}
  \sum_{p\le N}\gamma^p(x)
  &=
  \sum_{m=2}^N a_mb_m\\
  &=
  \frac{T_N(x)}{\log N}
  +
  \sum_{n=2}^{N-1}
  T_n(x)
  \left(
    \frac{1}{\log n}
    -
    \frac{1}{\log(n+1)}
  \right)\\
  &=
  \frac{N}{\log N}B_N(x)
  +
  \sum_{n=2}^{N-1}
  n
  \left(
    \frac{1}{\log n}
    -
    \frac{1}{\log(n+1)}
  \right)
  B_n(x).
\end{aligned}
\end{equation}

Dividing \eqref{eq:abel-unweighted-from-weighted} by $\pi(N)=|P_N|$ gives
\begin{equation}\label{eq:AN-positive-combination-Bn}
  A_N(x)
  =
  \alpha_{N,N}B_N(x)
  +
  \sum_{n=2}^{N-1}\alpha_{N,n}B_n(x),
\end{equation}
where
\begin{equation*}
  \alpha_{N,N}
  :=
  \frac{N}{\pi(N)\log N}
\end{equation*}
and, for $2\le n<N$,
\begin{equation*}
  \alpha_{N,n}
  :=
  \frac{n}{\pi(N)}
  \left(
    \frac{1}{\log n}
    -
    \frac{1}{\log(n+1)}
  \right).
\end{equation*}
Since $t\mapsto1/\log t$ is decreasing on $(1,\infty)$, all the
coefficients in \eqref{eq:AN-positive-combination-Bn} are nonnegative.

We next show that their sum is bounded uniformly in $N$. A direct calculation gives
\begin{equation}\label{eq:abel-coefficient-sum-telescoping}
\begin{aligned}
  \frac{N}{\log N}
  +
  \sum_{n=2}^{N-1}
  n\left(
    \frac{1}{\log n}
    -
    \frac{1}{\log(n+1)}
  \right)=
  \frac{2}{\log2}
  +
  \sum_{n=3}^N\frac{1}{\log n}.
\end{aligned}
\end{equation}
Combining \eqref{eq:abel-coefficient-sum-telescoping} with the first
estimate in \eqref{eq:reciprocal-log-estimates} and the lower Chebyshev bound
in \eqref{eq:chebyshev-pi-bounds}, we obtain
\begin{equation}\label{eq:abel-coefficients-uniform-sum}
  \alpha_{N,N}
  +
  \sum_{n=2}^{N-1}\alpha_{N,n}
  \le
  C,
  \qquad N\ge2.
\end{equation}

Since $x\ge0$, $B_n(x)\ge0$ for every
$n\ge2$. By \Cref{thm:weighted-prime-maximal-Lp},
we have
\begin{equation*}
  \bigl\|(B_n(x))_{n\ge2}\bigr\|_{
    L_r(\Ncal;\ell_\infty)}
  \le
  C_r\norm{x}_{L_r(\Ncal)}.
\end{equation*}
Let $\varepsilon>0$. By the positive-majorant characterization of
$L_r(\Ncal;\ell_\infty)$, there exists
$y_\varepsilon\in L_r(\Ncal)_+$ such that
\begin{equation}\label{eq:BN-positive-majorant}
  B_n(x)\le y_\varepsilon,
  \qquad n\ge2,
\end{equation}
and
\begin{equation*}
  \norm{y_\varepsilon}_{L_r(\Ncal)}
  \le
  \bigl\|(B_n(x))_{n\ge2}\bigr\|_{
    L_r(\Ncal;\ell_\infty)}
  +\varepsilon.
\end{equation*}
Using the positivity of the coefficients in
\eqref{eq:AN-positive-combination-Bn}, together with
\eqref{eq:abel-coefficients-uniform-sum} and
\eqref{eq:BN-positive-majorant}, we obtain
\begin{equation*}
  0\le A_N(x)
  \le
  \left(
    \alpha_{N,N}
    +
    \sum_{n=2}^{N-1}\alpha_{N,n}
  \right)y_\varepsilon
  \le
  Cy_\varepsilon,
  \qquad N\ge2.
\end{equation*}
Therefore,
\begin{equation*}
\begin{aligned}
  \bigl\|(A_N(x))_{N\ge2}\bigr\|_{
    L_r(\Ncal;\ell_\infty)}
  &\le
  C_r\norm{y_\varepsilon}_{L_r(\Ncal)}\\
  &\le
  C_r\bigl\|(B_n(x))_{n\ge2}\bigr\|_{
    L_r(\Ncal;\ell_\infty)}
  +C_r\varepsilon\\
  &\le
  C_r\norm{x}_{L_r(\Ncal)}+C_r\varepsilon.
\end{aligned}
\end{equation*}
Letting $\varepsilon\to0$ proves
\begin{equation}\label{eq:AN-maximal-positive-final}
  \bigl\|(A_N(x))_{N\ge2}\bigr\|_{
    L_r(\Ncal;\ell_\infty)}
  \le
  C_r\norm{x}_{L_r(\Ncal)}
\end{equation}
for every $x\in L_r(\Ncal)_+$.

We finally consider an arbitrary $x\in L_r(\Ncal)$. Write
\begin{equation*}
  x
  =
  (\Re x)_+-(\Re x)_-
  +i(\Im x)_+-i(\Im x)_-.
\end{equation*}
Applying \eqref{eq:AN-maximal-positive-final} to each of these four
positive elements and using the triangle inequality in
$L_r(\Ncal;\ell_\infty)$, we obtain
\begin{equation*}
\begin{aligned}
  \bigl\|(A_N(x))_{N\ge2}\bigr\|_{
    L_r(\Ncal;\ell_\infty)}
  &\le
  C_r\Bigl(
    \norm{(\Re x)_+}_{L_r(\Ncal)}
    +\norm{(\Re x)_-}_{L_r(\Ncal)}\\
  &\hspace{3.7em}
    +\norm{(\Im x)_+}_{L_r(\Ncal)}
    +\norm{(\Im x)_-}_{L_r(\Ncal)}
  \Bigr)\\
  &\le
  C_r\norm{x}_{L_r(\Ncal)}.
\end{aligned}
\end{equation*}

\subsection{Equivalence of the weighted and unweighted limits}
\label{sec:proof-weighted-unweighted-equivalence}

In this subsection, we prove that for $x\in L_r(\mathcal N)$, the unweighted prime averages $(A_N(x))$ and the logarithmically weighted prime averages $(B_N(x))$
have the same bilaterally almost uniform limit. 
More precisely, the convergence part of of
\Cref{thm:prime-bau-convergence} follows from 
\Cref{thm:weighted-prime-maximal-Lp} and the following lemma.

\begin{lemma}
\label{thm:weighted-unweighted-equivalence}
Let $1<p<\infty$ and $x\in L_p(\mathcal{N})$. Then $A_N(x)$ converges b.a.u. to some $y\in L_p(\mathcal{N})$ if and only if $B_N(x)$ converges b.a.u. to the same element $y$.
\end{lemma}

In the commutative setting, Wierdl proved the equivalence of the weighted
and unweighted limits by a pointwise summation-by-parts argument; see
\cite[Lemma~1]{Wierdl}.  This argument cannot be used directly for b.a.u.
convergence.  We instead use the following b.a.u. version of the classical
Silverman--Toeplitz theorem.  It gives a simple and unified proof in both
directions.

\begin{lemma}\label{lem:bau-toeplitz}
Let $(x_n)_{n\ge1}\subset L_0(\Ncal)$ and
$x\in L_0(\Ncal)$. Suppose that
\[
  x_n\longrightarrow x
  \qquad\text{b.a.u.}
\]
Let $(a_{N,n})_{N,n\ge1}$ be a scalar triangular matrix; that is,
$a_{N,n}=0$ whenever $n>N$. Assume that
\begin{equation}\label{eq:toeplitz-column-condition}
  \lim_{N\to\infty}a_{N,n}=0
  \qquad\text{for every fixed }n,
\end{equation}
\begin{equation}\label{eq:toeplitz-row-sum-condition}
  \lim_{N\to\infty}\sum_{n=1}^N a_{N,n}=1,
\end{equation}
and
\begin{equation}\label{eq:toeplitz-absolute-row-bound}
  M
  :=
  \sup_{N\ge1}\sum_{n=1}^N|a_{N,n}|
  <\infty.
\end{equation}
Then
\begin{equation}\label{eq:toeplitz-bau-conclusion}
  \sum_{n=1}^N a_{N,n}x_n
  \longrightarrow x
  \qquad\text{b.a.u..}
\end{equation}
\end{lemma}

\begin{proof}
Fix $\varepsilon>0$. Since $x_n\to x$ b.a.u., there exists a
projection $e_0\in\Ncal$ such that
\begin{equation*}
  \tau(e_0^\perp)<\frac{\varepsilon}{2}
  \qquad\text{and}\qquad
  \lim_{n\to\infty}\|e_0(x_n-x)e_0\|_\infty=0.
\end{equation*}
On the other hand, we may choose
projections $f_0,f_1,f_2,\ldots\in\Ncal$ such that
\begin{equation*}
  \tau(f_0^\perp)<\frac{\varepsilon}{8},
  \qquad
  \tau(f_n^\perp)<\frac{\varepsilon}{2^{n+3}}
  \quad(n\ge1),
\end{equation*}
and
\begin{equation*}
  \|f_0xf_0\|_\infty<\infty,
  \qquad
  \|f_n(x_n-x)f_n\|_\infty<\infty
  \quad(n\ge1).
\end{equation*}
Set
\begin{equation*}
  e
  :=
  e_0\wedge\bigwedge_{n=0}^{\infty}f_n.
\end{equation*}
Then,
\begin{equation}\label{eq:toeplitz-common-projection-trace}
\begin{aligned}
  \tau(e^\perp)
  \le
  \tau(e_0^\perp)
  +
  \sum_{n=0}^{\infty}\tau(f_n^\perp)
  <\varepsilon;
\end{aligned}
\end{equation}
moreover
\begin{equation}\label{eq:toeplitz-common-boundedness}
  \|exe\|_\infty<\infty,
  \qquad
  \|e(x_n-x)e\|_\infty<\infty
  \quad(n\ge1),
\end{equation}
and, because $e\le e_0$,
\begin{equation}\label{eq:toeplitz-common-tail-limit}
  \lim_{n\to\infty}\|e(x_n-x)e\|_\infty=0.
\end{equation}

Set
\[
  y_N:=\sum_{n=1}^N a_{N,n}x_n
  \qquad\text{and}\qquad
  r_N:=\sum_{n=1}^N a_{N,n}.
\]
Then
\begin{equation}\label{eq:toeplitz-error-decomposition}
  y_N-x
  =
  \sum_{n=1}^N a_{N,n}(x_n-x)
  +(r_N-1)x.
\end{equation}
Let $\delta>0$. By \eqref{eq:toeplitz-common-tail-limit}, choose $m$
so large that
\begin{equation}\label{eq:toeplitz-small-tail}
  \sup_{n>m}\|e(x_n-x)e\|_\infty
  <
  \frac{\delta}{3M}.
\end{equation}
For this fixed $m$, compressing
\eqref{eq:toeplitz-error-decomposition} by $e$ gives
\begin{equation*}
\begin{aligned}
  \|e(y_N-x)e\|_\infty
  &\le
  \sum_{n=1}^m
  |a_{N,n}|\,
  \|e(x_n-x)e\|_\infty\\
  &\quad+
  \sum_{n=m+1}^N
  |a_{N,n}|\,
  \|e(x_n-x)e\|_\infty\\
  &\quad+
  |r_N-1|\,\|exe\|_\infty.
\end{aligned}
\end{equation*}
By \eqref{eq:toeplitz-column-condition} and
\eqref{eq:toeplitz-common-boundedness}, the first term on the
right-hand side tends to zero as $N\to\infty$, since it is a finite
sum. By \eqref{eq:toeplitz-absolute-row-bound} and
\eqref{eq:toeplitz-small-tail}, the second term is bounded by
$\delta/3$. The third term tends to zero by
\eqref{eq:toeplitz-row-sum-condition} and
\eqref{eq:toeplitz-common-boundedness}. Hence, for all sufficiently
large $N$,
\[
  \|e(y_N-x)e\|_\infty<\delta.
\]
Since $\delta>0$ is arbitrary and the projection $e$ is independent of
$\delta$, we conclude that
\[
  \lim_{N\to\infty}\|e(y_N-x)e\|_\infty=0.
\]
Together with \eqref{eq:toeplitz-common-projection-trace}, this proves
\eqref{eq:toeplitz-bau-conclusion}.
\end{proof}

Now, we are ready to show \Cref{thm:weighted-unweighted-equivalence}.
\begin{proof}[Proof of
\Cref{thm:weighted-unweighted-equivalence}]
We first prove that b.a.u. convergence of $(B_N(x))_{N\ge2}$ implies
b.a.u. convergence of $(A_N(x))_{N\ge2}$ to the same limit. As in showing the maximal inequalities,
we start with the identity \eqref{eq:AN-positive-combination-Bn}, that is,
\begin{equation}\label{eq:AN-as-Toeplitz-transform-of-BN}
  A_N(x)
  =
  \sum_{n=2}^N\alpha_{N,n}B_n(x),
\end{equation}
where
\begin{equation*}
  \alpha_{N,n}
  :=
  \begin{cases}
    \displaystyle
    \frac{n}{\pi(N)}
    \left(
      \frac{1}{\log n}
      -
      \frac{1}{\log(n+1)}
    \right),
      & 2\le n<N,\\[1.2em]
    \displaystyle
    \frac{N}{\pi(N)\log N},
      & n=N,\\[0.8em]
    0,
      & n>N.
  \end{cases}
\end{equation*}
To get the desired assertion, it suffices to verify the hypotheses of \Cref{lem:bau-toeplitz}.

First, \eqref{eq:toeplitz-column-condition} is trivial, and
\eqref{eq:toeplitz-absolute-row-bound} is just \eqref{eq:abel-coefficients-uniform-sum} since $\alpha_{N,n}\geq0$. It remains to justify \eqref{eq:toeplitz-row-sum-condition}.
A direct calculation yields
\begin{equation}\label{eq:alpha-row-sum-exact}
\begin{aligned}
  \sum_{n=2}^N\alpha_{N,n}
  =
  \frac{1}{\pi(N)}
  \left(
    \frac{2}{\log2}
    +
    \sum_{n=3}^N\frac{1}{\log n}
  \right).
\end{aligned}
\end{equation}
Set
\[
  C_N:=\sum_{n=3}^N\frac1{\log n}.
\]
The two precise asymptotics
\[
  \frac{C_N}{N/\log N}\longrightarrow1
  \quad\text{and}\quad
  \frac{\pi(N)}{N/\log N}\longrightarrow1
\]
follow from \eqref{eq:reciprocal-log-estimates} and
\eqref{eq:PNT-pi-form}, respectively.  Hence the leading constants
cancel, while the fixed boundary term in
\eqref{eq:alpha-row-sum-exact} is negligible.  More explicitly,
\begin{equation*}
\begin{aligned}
  \lim_{N\to\infty}\sum_{n=2}^N\alpha_{N,n}
  &=
  \lim_{N\to\infty}
  \left[
    \frac{2}{\pi(N)\log2}
    +
    \frac{C_N}{N/\log N}
    \frac{N/\log N}{\pi(N)}
  \right]\\
  &=0+1\cdot1=1,
\end{aligned}
\end{equation*}
which is \eqref{eq:toeplitz-row-sum-condition}.

\medskip

We now prove the converse implication. The argument is similar. For $n\ge2$, set
\begin{equation*}
  S_n(x)
  :=
  \sum_{p\le n}\gamma^p(x)
  =
  \pi(n)A_n(x).
\end{equation*}
Applying discrete Abel summation with the weight $\log n$ gives
\begin{equation}\label{eq:BN-from-AN-Abel}
\begin{aligned}
  \sum_{p\le N}(\log p)\gamma^p(x)
  &=
  (\log N)S_N(x)
  -
  \sum_{n=2}^{N-1}
  S_n(x)
  \bigl(\log(n+1)-\log n\bigr)\\
  &=
  \pi(N)(\log N)A_N(x)
  -
  \sum_{n=2}^{N-1}
  \pi(n)
  \bigl(\log(n+1)-\log n\bigr)A_n(x).
\end{aligned}
\end{equation}
After dividing by $N$, we obtain
\begin{equation}\label{eq:BN-as-Toeplitz-transform-of-AN}
  B_N(x)
  =
  \sum_{n=2}^N\beta_{N,n}A_n(x),
\end{equation}
where
\begin{equation}\label{eq:beta-coefficients}
  \beta_{N,n}
  :=
  \begin{cases}
    \displaystyle
    -\frac{\pi(n)}{N}
    \bigl(\log(n+1)-\log n\bigr),
      & 2\le n<N,\\[1.2em]
    \displaystyle
    \frac{\pi(N)\log N}{N},
      & n=N,\\[0.8em]
    0,
      & n>N.
  \end{cases}
\end{equation}
Again, we verify the hypotheses of \Cref{lem:bau-toeplitz}.

For every fixed $n\ge2$, once $N>n$,
\[
  |\beta_{N,n}|
  =
  \frac{\pi(n)}{N}
  \bigl(\log(n+1)-\log n\bigr),
\]
so
\begin{equation}\label{eq:beta-fixed-column-limit}
  \lim_{N\to\infty}\beta_{N,n}=0.
\end{equation}

Applying \eqref{eq:BN-from-AN-Abel} to the constant scalar sequence
$A_n=1$, or simply repeating the same calculation, gives
\begin{equation*}
  \sum_{n=2}^N\beta_{N,n}
  =
  \frac{1}{N}
  \sum_{p\le N}\log p
  =
  \frac{\vartheta(N)}{N}.
\end{equation*}
 Therefore, \eqref{eq:PNT-pi-form} gives
\begin{equation}\label{eq:beta-row-sum-limit}
  \lim_{N\to\infty}\sum_{n=2}^N\beta_{N,n}
  =
  \lim_{N\to\infty}\frac{\vartheta(N)}{N}
  =1.
\end{equation}

It remains to verify the uniform bound for the absolute row sums. By
\eqref{eq:beta-coefficients},
\begin{equation*}
\begin{aligned}
  \sum_{n=2}^N|\beta_{N,n}|
  &=
  \frac{\pi(N)\log N}{N}
  +
  \frac{1}{N}
  \sum_{n=2}^{N-1}
  \pi(n)\bigl(\log(n+1)-\log n\bigr).
\end{aligned}
\end{equation*}
Moreover, by the upper bound in \eqref{eq:chebyshev-pi-bounds},
\begin{equation*}
  \pi(n)\le C\frac{n}{\log n},
  \qquad n\ge2,
\end{equation*}
and
\begin{equation*}
  \log(n+1)-\log n
  =
  \log\left(1+\frac{1}{n}\right)
  \le
  \frac{1}{n}.
\end{equation*}
Consequently,
\begin{equation*}
\begin{aligned}
  \sum_{n=2}^N|\beta_{N,n}|
  &\le
  C
  +
  \frac{C}{N}
  \sum_{n=2}^{N-1}\frac{1}{\log n}\\
  &\le
  C,
\end{aligned}
\end{equation*}
where in the last step we used \eqref{eq:reciprocal-log-estimates}. Hence
\begin{equation}\label{eq:beta-absolute-row-bound}
  \sup_{N\ge2}
  \sum_{n=2}^N|\beta_{N,n}|
  <\infty.
\end{equation}

Suppose that
\[
  A_N(x)\longrightarrow y
  \qquad\text{b.a.u.}
\]
Then \eqref{eq:BN-as-Toeplitz-transform-of-AN}, together with
\eqref{eq:beta-fixed-column-limit},
\eqref{eq:beta-row-sum-limit},
\eqref{eq:beta-absolute-row-bound}, and
\Cref{lem:bau-toeplitz}, gives
\begin{equation*}
  B_N(x)\longrightarrow y
  \qquad\text{b.a.u.}
\end{equation*}
This proves both implications and completes the proof.
\end{proof}

{\it Declaration on the use of generative AI}: During the preparation of this manuscript, the authors used 
GPT-5.6 Pro  only to assist with English-language
editing, the organization and presentation of the manuscript,
\LaTeX{} formatting, and preliminary checks of mathematical
consistency. All mathematical ideas, results, arguments, and proofs
were developed and independently verified by the authors. The authors
carefully reviewed and edited all AI-assisted suggestions and take full
responsibility for the accuracy, originality, and content of the
manuscript.

\end{document}